\documentclass[11pt]{amsart}   
\usepackage[greek,english]{babel}  
\usepackage{graphicx}
\usepackage[pdftex,bookmarks]{hyperref}
\usepackage{amsmath}
\usepackage{amssymb}
\usepackage{amsmath,amsthm}
\usepackage{verbatim}
\usepackage{sidecap}
\usepackage{tikz}
\newtheorem{theorem}{Theorem}[section]

\newtheorem{lemma}[theorem]{Lemma}

\newtheorem{proposition}[theorem]{Proposition}

\newtheorem{corollary}[theorem]{Corollary}

\theoremstyle{definition}
\newtheorem{remark}[theorem]{Remark}

\newcommand{\sgn}{\textnormal{sgn}}

\titlepage 
\title{Scattering parabolic solutions for the spatial $N$-centre problem}
\author{Alberto Boscaggin, Walter Dambrosio and Susanna Terracini}
\address{Alberto Boscaggin, Walter Dambrosio and Susanna Terracini \newline \indent
 Dipartimento di Matematica ``Giuseppe Peano'', \newline \indent
Universit\`a di Torino, \newline \indent
Via Carlo Alberto, 10,
10123 Torino, Italy}

\email{alberto.boscaggin@unito.it}
\email{walter.dambrosio@unito.it}
\email{susanna.terracini@unito.it}
\date{}

\begin{document}

\begin{abstract}
For the $N$-centre problem in the three dimensional space,
$$
\ddot x = -\sum_{i=1}^{N} \frac{m_i \,(x-c_i)}{\vert x - c_i \vert^{\alpha+2}}, \qquad x \in \mathbb{R}^3 \setminus \{c_1,\ldots,c_N\},
$$
where $N \geq 2$, $m_i > 0$ and $\alpha \in [1,2)$, we prove the existence of entire parabolic trajectories having prescribed asymptotic directions.
The proof relies on a variational argument of min-max type.
Morse index estimates and regularization techniques are used in order to rule out the possible occurrence of collisions.  
\end{abstract}
\date{\today}

\thanks{{\bf Acknowlegments.}  The authors wish to thank Prof. Rafael Ortega for some useful discussions.  
\smallbreak
\noindent
Work partially supported by the 
ERC Advanced Grant 2013 n. 339958
{\it Complex Patterns for Strongly Interacting Dynamical Systems - COMPAT}, by the PRIN-2012-74FYK7 Grant {\it Variational and perturbative aspects of nonlinear differential problems}; A. Boscaggin and W. Dambrosio were also supported by the GNAMPA Project 2015 {\it Equazioni Differenziali
Ordinarie sulla retta reale}.
}

\maketitle
\medbreak

\section{Introduction and statement of the main result}
\setcounter{page}{1}

The $N$-centre problem is a simplified version of the restricted circular $N+1$-body problem in a rotating frame, where the centrifugal force is neglected; it concerns the
motion of a point mass moving under the attraction due to $N$ fixed centers of force $c_1,\dots,c_N$. 
In this paper we shall be concerned with homogeneous potential of degree
$-\alpha$, with $\alpha\in[1,2)$, thus including the newtonian gravitational case ($\alpha=1$), in the \emph{three dimensional space}. So the motion equation takes the form
\begin{equation}\label{main}
\ddot x = -\sum_{i=1}^{N} \frac{m_i \,(x-c_i)}{\vert x - c_i \vert^{\alpha+2}} , \qquad x \in \mathbb{R}^3 \setminus \{c_1,\ldots,c_N\},
\end{equation}
where $N \geq 2$, $m_i > 0$, $c_i \in \mathbb{R}^3$ (with $c_i \neq c_j$ for $i \neq j$) 
%and $\alpha \in [1,2)$, 
and 
the associated Hamiltonian is
$$
H(p,x)=\frac{1}{2}\vert p \vert^2 - \sum_{i=1}^N \frac{m_i}{\alpha \vert x- c_i \vert^\alpha}\;.
$$

Our aim is to prove the existence of unbounded non-collision entire trajectories having zero energy (i.e., \emph{parabolic trajectories}) and prescribed ingoing and outgoing directions.  In spite of their natural structural instability, these  orbits act as connections between different normalized configurations and can be used as \emph{carriers from one to the other region of the phase space}; as such, they have been used as building blocks for constructing complex trajectories (see, e.g. \cite{SoaTer12}). In recent papers \cite{BarTerVer14,BarTerVer2013, TeVe} the existence of parabolic trajectories has been considered for the 
anisotropic Kepler problem and for the $N$-body problem; more precisely, the presence of parabolic orbits and their variational character has been linked with the existence of minimal collision trajectories and eventually with the detection of unbounded families of non-collision periodic orbits \cite{LuzMad14,MadVen09,SoaTer12}. Non-trivial parabolic orbits may be of interest also from the point of view of the applications of weak KAM theory in Celestial Mechanics; indeed, since they are homoclinic to the infinity, which represents the Aubry-Mather set of our system, they can be used to construct multiple viscosity solutions of the associated Hamilton-Jacobi equation (see also \cite{Mad14}). 

We are going to prove the following result.

\begin{theorem}\label{thmain}
For any
$\xi^+, \xi^- \in \mathbb{S}^2$ with $\xi^+ \neq \xi^-$, there exists a spatial parabolic solution $x: \mathbb{R} \to \mathbb{R}^3 \setminus \{ c_1,\ldots,c_N\}$
of \eqref{main} such that
$$
\vert x(t) \vert \sim \left(\sqrt{\frac{m}{2\alpha}} (2+\alpha)\right)^{\frac{2}{2+\alpha}} \, \vert t \vert^{\frac{2}{2+\alpha}}, \qquad t \to\pm\infty,
$$
and
\begin{equation}\label{asintoticast}
\lim_{t \to \pm \infty}\frac{x(t)}{\vert x(t) \vert} = \xi^{\pm}. 
\end{equation}
%When $\xi^+=\xi^-$, there exists a one-collision spatial parabolic solution $x: \mathbb{R} \to \mathbb{R}^3$ satisfying \eqref{asintoticastr} and \eqref{asintoticast} with $\xi^+ = \xi^-$.
\end{theorem}

We remark that, when $\xi^+ = \xi^-$, we can still ensure the existence of a generalized spatial parabolic solutions of \eqref{main} satisfying 
\eqref{asintoticast} (for $\xi^+ = \xi^-$), having maybe some collisions with the set of the centers (see Remark \ref{xi}). 
Let us now examine our Theorem in the contest of scattering: the \emph{scattering angle} is that between the outgoing and incident directions. So Theorem
\ref{thmain} states the existence of at least one spatial trajectory having vanishing asymptotic velocity \emph{for every scattering angle}. Let us stress that this is not the case for central $-\alpha$-homogeneous potentials: for the Newtonian potential $1/r$, it is a straightforward consequence of the preservation of the Runge-Lenz vector that the only allowed scattering angle is $2\pi$. However, for potentials of the form $1/r^\alpha$ with $\alpha>1$, the parabolic trajectories form a loop, as the scattering angle can be shown to be $2\pi/(2-\alpha)>2\pi$ at zero energy (see Proposition \ref{propspan}).  This picture is in striking contrast with the positive energy case, where, for hyperbolic trajectories, all (but one) scattering angles are always achieved.
The presence of two or more centers results into the occurrence of parabolic connections between every pair of asymptotic configurations, thus allowing every value of the scattering angle, similarly with the hyperbolic case
 \cite{FelTan00,FelTan000}.

In the planar $N$-centre problem, unbounded non-collision parabolic trajectories are known to exist in various homotopy classes of paths and the zero energy shell exhibits a symbolic dynamics (see e.g. \cite{KleKna92}). Indeed, planar unbounded parabolic trajectories can be symbolically described by their topological properties. They are all local minimizers for the action and the Jacobi metric. In contrast, local action minimizing unbounded parabolic trajectories are not expected to exist in the three dimensional space. The ultimate reason rests in the properties of the scattering angle: very interestingly, the ``looping'' occurring for $\alpha\geq 1$ has been linked in \cite{Tan93b} by K. Tanaka with a change in the Morse index of the parabolic solutions (see also \cite{Tan94,Tan93}). Similarly to the case of unbounded hyperbolic trajectories \cite{FelTan00,FelTan000} our solutions too will have a nontrivial Morse index, as it will result as a mountain pass variational argument: the absence of collisions will be related with the Morse index. 
Notice, however, that the case $\alpha = 1$ is particularly delicate and indeed it requires an additional analysis, based on regularization techniques 
(see \cite{Spe69}).

It is well known since the times of Euler (in 1760) that the planar two-centre problem can be integrated by using elliptic-hyperbolic coordinates (see e.g. \cite{Whittaker}). The planar case of $N$-centre with
$N\geq 3$ is known to be non integrable on non-negative energy levels and has positive entropy; some partial extensions are available also for the spatial case (see \cite{Bol84,BolNeg03,BolNeg01,Dim10,KleKna92,Kna02}). Recently, in \cite{SoaTer12} Soave and Terracini have shown the presence of a chaotic subsystem for the planar $N$-centre problem also at negative energies.
Let us finally mention that topologically nontrivial periodic trajectories have been recently investigated both in the planar and spatial $N$-body and the $N$-centre problems by means of constrained minimization  arguments (see \cite{Cas09,Chen08, ChenYu2015,DulMon,FuGrNe,SoaTer13}).

This paper is organized in five Sections and one Appendix. Section \S2 is devoted to investigate the general properties of parabolic solutions. 
In Section \S 3 we show how to approximate entire solutions to \ref{main}, by considering the finite time interval auxiliary problem. In Section \S 4 we set up a min-max
scheme and we show  basic estimates for the critical values and the corresponding solutions, with attention to their Morse index. Finally, in Section \S 5 we study
some properties of the approximate solutions in order to control their behavior at infinity and rule out the presence of collisions, completing the proof of our Theorem. The Appendix is devoted to a systematic study of parabolic arcs of the fully $-\alpha$-homogenous case and of 
their variational characterizations.
%\medbreak
%\noindent
%\textbf{Notation.} 
\subsection{Notation}

The symbols $x \cdot y$ and $\vert x \vert$ denote the standard Euclidean product and Euclidean norm on $\mathbb{R}^3$,
$B_\rho(x)$ is the open ball of radius $\rho$ centered at $x$. The symbols $\langle u, v \rangle$ and $\Vert u \Vert$ stand for the usual scalar product and the associated norm on the Sobolev space $H^1([a,b]; \mathbb{R}^3)$, namely 
$$
\langle u,v \rangle = \int_a^b \left( u(t) \cdot v(t) + \dot u(t) \cdot \dot v(t) \right)\,dt,
\qquad
\Vert u \Vert = \left[\int_a^b \left(\vert u(t) \vert^2 + \vert \dot u(t) \vert^2 \right)\,dt\right]^{1/2}.
$$ 
Finally, $\textnormal{j}(A)$ is the Morse-index of a self-adjoint bounded linear operator $A$ on an Hilbert space.

\subsection{Technical estimates on the potential}\label{stimepot}

Let us define the collision set
$$
\Sigma = \{ c_1,\ldots,c_N\}
$$
and the potential 
\begin{equation}\label{defv}
V(x) = \sum_{i=1}^N \frac{m_i}{\alpha \vert x - c_i \vert^\alpha}, \qquad x \in \mathbb{R}^3 \setminus \Sigma.
\end{equation}
Also, let 
\begin{equation}\label{defmassa}
m = \sum_{i=1}^N m_i, \qquad \Xi = \max_i \vert c_i \vert.
%\qquad  = \sum_{i=1}^N m_i \vert c_i \vert^2,  
\end{equation}
Without loss of generality, 
we finally assume that the center of mass is placed at the origin, namely
\begin{equation}\label{centromassa}
\sum_{i=1}^N m_i c_i = 0.
\end{equation}
Throughout the paper, both the behavior of $V$ near the centers $c_i$ and the behavior of $V$ at infinity will play
an important role. Hence, we fix here some useful notation.

As for the behavior of $V$ near the singularities, for any $i=1,\ldots,N$ we write
\begin{equation}\label{vsing}
V(x) = \frac{m_i}{\alpha \vert x - c_i \vert^{\alpha}} + \Phi_i(x).
\end{equation}
Of course, $\Phi_i \in \mathcal{C}^\infty(\mathbb{R}^3 \setminus (\Sigma \setminus \{c_i\}))$.
From now on, we choose a constant $\delta^* > 0$ so small that 
\begin{equation}\label{deltastar}
B_{\delta^*}(c_i) \subset B_{\Xi + 1}(0), \; \forall\, i=1,\ldots,N, \qquad B_{\delta^*}(c_i) \cap B_{\delta^*}(c_j) = \emptyset, \;\forall \,i \neq j.
\end{equation}
%where we have used the notation $B_\rho(x)$ for the open ball of radius $\rho > 0$ centered at $x$.
Moreover, we also assume
\begin{equation}\label{deltastar2}
\frac{2-\alpha}{\alpha}\frac{m_i}{\vert x - c_i \vert^{\alpha}}+2 \Phi_i(c_i) + \nabla\Phi_i(x) \cdot (x-c_i) > 0, \quad \mbox{ for } 0 < \vert x - c_i \vert \leq \delta^*, 
\end{equation}
and
\begin{equation}\label{deltastar3}
V(x) \leq \frac{3m_i}{2\alpha \vert x - c_i \vert^\alpha}, \quad \mbox{ for } 0 < \vert x - c_i \vert \leq \delta^*, 
\end{equation}
for $i=1,\ldots,N$.

On the other hand, dealing with the behavior of $V$ at infinity, we set
\begin{equation}\label{vinf}
V(x) = \frac{m}{\alpha \vert x \vert^{\alpha}} + W(x).
\end{equation}
Using \eqref{centromassa}, we can easily see that
$$
W(x) = O\left( \frac{1}{\vert x \vert^{\alpha+2}}\right) \quad \mbox{ and } \quad
\nabla W(x) = O\left( \frac{1}{\vert x \vert^{\alpha+3}}\right), \quad
\mbox{ for } \vert x \vert \to +\infty.
$$
As a consequence, we can fix constants $C_-, C_+ > 0$ and $K > \Xi + 1$ such that
\begin{equation}\label{stimaW}
\vert W(x) \vert \leq \frac{C_+}{\vert x \vert^{\alpha+2}} \quad \mbox{ and }
\quad \vert \nabla W(x) \vert \leq \frac{C_+}{\vert x \vert^{\alpha+3}}, \quad \mbox{ for every } \vert x \vert \geq K,
\end{equation}
\begin{equation}\label{stimaW2}
2 \vert W(x) \vert + \vert \nabla W(x) \cdot x \vert \leq \frac{(2-\alpha) m}{4 \alpha} \frac{1}{\vert x \vert^{\alpha}}, \quad \mbox{ for every } \vert x \vert \geq K,
\end{equation}
\begin{equation}\label{stimaV}
\frac{C_-}{\vert x \vert^{\alpha}} \leq V(x) \leq \frac{C_+}{\vert x \vert^{\alpha}}, \quad \mbox{ for every } \vert x \vert \geq K,
\end{equation}
and
\begin{equation}\label{stimak}
\sqrt{\frac{m}{\alpha}} \frac{1}{\vert x \vert^{\alpha/2}} - \frac{C_+}{\vert x \vert^{2+\alpha/2}} \leq 
\sqrt{V(x)} \leq
\sqrt{\frac{m}{\alpha}} \frac{1}{\vert x \vert^{\alpha/2}} + \frac{C_+}{\vert x \vert^{2+\alpha/2}}, \quad \mbox{ for every } \vert x \vert \geq K.
\end{equation}
The estimates \eqref{stimaW}, \eqref{stimaW2} and \eqref{stimak} are rather obvious, while \eqref{stimak} follows from \eqref{stimaW} using the elementary inequalities
$1-2\vert s \vert \leq \sqrt{1+s} \leq 1+\tfrac{1}{2}s$ (valid for $s \geq -1$).

\section{Some general properties of parabolic solutions}\label{sez2}

In this section we collect some general properties valid for ``large'' parabolic solutions of \eqref{main}.
More precisely, we deal with solutions $x:[t_1,t_2] \to \mathbb{R}^3$ of \eqref{main}, with
$-\infty \leq t_1 < t_2 \leq +\infty$ (in the case $t_i \in \{\pm \infty\}$, we obviously mean that 
$t_i$ is not included in the interval of definition of $x$), satisfying
the zero-energy relation
\begin{equation}\label{0en}
\frac{1}{2}\vert \dot x(t) \vert^2 - \sum_{i=1}^N \frac{m_i}{\alpha \vert x(t) - c_i \vert^\alpha} = 0, \quad \mbox{ for every } t \in I,
\end{equation}
and
\begin{equation}\label{largep}
\vert x(t) \vert \geq K, \quad \mbox{ for every } t \in [t_1,t_2],
\end{equation}
where $K > \Xi +1$ is the constant fixed in Subsection \ref{stimepot}.
Due to this last assumption, we always write
\begin{equation}\label{polar1}
x(t) = r(t)s(t)
\end{equation}
with $r(t) = \vert x(t) \vert \geq K$ and $s(t) = \tfrac{x(t)}{\vert x(t) \vert} \in \mathbb{S}^2$. 
In these new coordinates, recalling the definition \eqref{defv}, the fact that $x$ has zero energy reads as
\begin{equation}\label{eqenergia}
\dot r^2 + r^2 \vert \dot s \vert^2 = 2V(rs),
\end{equation}
while the differential equation \eqref{main} becomes
\begin{equation}\label{eqpolare}
\ddot r = r \vert \dot s \vert^2 +  \nabla V(rs) \cdot s, \qquad
r \ddot s = \nabla_{\mathbb{S}^2} V(rs) - r \vert \dot s \vert^2 s - 2 \dot r \dot s,
\end{equation}
where $\nabla_{\mathbb{S}^2} V(rs) = \nabla V(rs) - ( \nabla V(rs) \cdot s) s$.
\smallbreak
As a first step, we define 
$$
I(t) = \frac{1}{2}\vert x(t) \vert^2 = \frac{1}{2} r^2(t), \quad \mbox{ for every } t \in [t_1,t_2],
$$
and we establish a Lagrange-Jacobi inequality. 

\begin{lemma}\label{lemlj}
Let $x: [t_1,t_2] \to \mathbb{R}^3$ be a parabolic solution of \eqref{main} satisfying \eqref{largep}.
Then
\begin{equation}\label{stimajacobi}
\ddot I(t) \geq \frac{(2-\alpha)m}{2\alpha r^\alpha(t)}, \quad \mbox{ for every } t \in [t_1,t_2].
\end{equation}
As a consequence, either $r$ is strictly monotone on $[t_1,t_2]$ or 
there exists $t^* \in (t_1,t_2)$ such that $r$ is strictly decreasing on $[t_1,t^*)$ and
strictly increasing on $(t^*,t_2]$.
\end{lemma}

\begin{proof}
A simple computation, based on \eqref{eqenergia} and \eqref{eqpolare}, shows that
\begin{equation}\label{lj}
\ddot I(t) = 2 V(x(t)) + \nabla V(x(t)) \cdot x(t), \quad \mbox{ for every } t \in [t_1,t_2].
\end{equation}
Using \eqref{vinf}, we thus find
$$
\ddot I(t) = \frac{(2-\alpha)m}{\alpha r^\alpha(t)} + 2 W(x) + \nabla W(x(t)) \cdot x(t), \quad \mbox{ for every } t \in [t_1,t_2],
$$
and we conclude in view of \eqref{stimaW2}.
\end{proof}

As a quite direct consequence of Lemma \ref{lemlj} we can also establish the following useful corollary, which will be used various times 
in the paper.

\begin{corollary}\label{cortempi}
Let $x: [t_1,t_2] \to \mathbb{R}^3$ be a parabolic solution of \eqref{main} satisfying \eqref{largep} and assume that 
$r$ is strictly monotone on the whole $[t_1,t_2]$. Then
\begin{equation}\label{fa}
\frac{\left\vert r(t_2)^{1+\alpha/2} -
r(t_1)^{1+\alpha/2} \right\vert}{\left(1+\alpha/2 \right)\sqrt{2C_+}} \leq t_2 - t_1 \leq \sqrt{\frac{2\alpha}{(2-\alpha)m}}\max\{ r(t_2), r(t_1)\}^{1+\alpha/2},
\end{equation}
where $C_+ > 0$ is the constant fixed in Subsection \ref{stimepot}.
\end{corollary}

\begin{proof}
We give the proof when $r$ is strictly increasing (the other case being analogous). At first, notice that, in view of Lemma \ref{lemlj}, 
$\dot r(t_1) \geq 0$ and $\dot r(t) > 0$ for $t \in (t_1,t_2]$.
Using the fact that $x$ has zero-energy and \eqref{stimaV}, we find
$$
0 < \dot r(t) \leq \sqrt{2C_+} r(t)^{-\alpha/2}, \quad \mbox{ for every } t \in (t_1,t_2].
$$
Hence
\begin{align*}
t_2 - t_1 & = \int_{t_1}^{t_2} \frac{\dot r(t)}{\dot r(t)}\,dt \geq \frac{1}{\sqrt{2C_+}} \int_{t_1}^{t_2} \frac{\dot r(t)}{r(t)^{-\alpha/2}}\,dt \\
& = \frac{1}{\left(1+\alpha/2 \right)\sqrt{2C_+}}\left( r(t_2)^{1+\alpha/2} -
r(t_1)^{1+\alpha/2} \right),
\end{align*}
thus proving the estimate from below. 
On the other hand, using \eqref{stimajacobi} we find, for $t \in [t_1,t_2]$, 
$$
\dot I(t) \geq \frac{(2-\alpha)m}{2\alpha} \int_{t_1}^t r^{-\alpha}(s)\,ds \geq
\frac{(2-\alpha)m}{2\alpha} \frac{t-t_1}{r^\alpha(t_2)}.
$$
Integrating on $[t_1,t_2]$, we thus have
$$
\frac{1}{2}r^2(t_2) = I(t_2) \geq \frac{(2-\alpha)m}{4\alpha} \frac{(t_2-t_1)^2}{r^\alpha(t_2)},
$$
giving the estimate from above.
\end{proof}

The next result gives an estimate for the angular momentum
$$
A(t) = x(t) \wedge \dot x(t), \quad \mbox{ for every } t \in [t_1,t_2].
$$
Notice that, in the coordinates \eqref{polar1}, 
\begin{equation}\label{momentopolare}
\quad \vert A(t) \vert = r^2(t) \vert \dot s(t) \vert. 
\end{equation}

\begin{lemma}\label{proppara}
Let $x: [t_1,t_2] \to \mathbb{R}^3$ be a parabolic solution of \eqref{main} satisfying \eqref{largep}.
Then
\begin{equation}\label{stimaapunto}
\vert \dot A(t) \vert \leq \frac{C_+}{r^{\alpha+2}(t)}, \quad \mbox{ for every } t \in [t_1,t_2],
\end{equation}
where $C_+ > 0$ is the constant fixed in Subsection \ref{stimepot}.
\end{lemma}

\begin{proof}
Since
$$
\dot A(t) = x(t) \wedge \nabla V(x(t)) = x(t) \wedge \nabla W(x(t)), \quad \mbox{ for every } t \in [t_1,t_2], 
$$
the conclusion follows from \eqref{stimaW}.
\end{proof}

Now we are in position to prove the main result of this section, giving asymptotic estimates for large parabolic solutions defined on the half-line
$[t_1,+\infty)$. Of course, a symmetric result holds for solutions defined on $(-\infty,t_2]$.

\begin{proposition}\label{asintotica}
Let $x: [t_1,+\infty) \to \mathbb{R}^3$ be a parabolic solution of \eqref{main} such that $\vert x(t) \vert \geq K$ for any $t \in [t_1,+\infty)$.
Then
\begin{equation}\label{rasy}
r(t) \sim \gamma_{\alpha,m} \, t\,^{\frac{2}{2+\alpha}},    
\qquad t  \to +\infty,
\end{equation}
%\dot r(t) \sim \frac{2\gamma_{\alpha,m}}{2+\alpha}  \,t^{\frac{-\alpha}{2+\alpha}},
where
\begin{equation}\label{gamdef}
\gamma_{\alpha,m} = \left(\sqrt{\frac{m}{2\alpha}} (2+\alpha)\right)^{\frac{2}{2+\alpha}},
\end{equation}
and there exists $\xi \in \mathbb{S}^2$ such that
$$
\lim_{t \to +\infty} s(t) = \xi.
$$
\end{proposition}

\begin{proof}
From Lemma \ref{lemlj} we immediately deduce that $r(t) \to +\infty$ for $t \to +\infty$ (monotonically);
from this fact together with Corollary \ref{cortempi} we infer the existence of $T_1 > t_1$ such that 
\begin{equation}\label{stimar}
r^{\alpha+2}(t) \geq \frac{(2-\alpha)m}{2\alpha}(t-T_1)^2, \quad \mbox{ for every } t > T_1.
\end{equation}
Therefore $\vert \dot A \vert$ is integrable at infinity, so that
\begin{equation}\label{alimitato}
\limsup_{t \to +\infty} \vert A(t) \vert < \infty.
\end{equation}
Now, we define the function
\begin{equation}\label{gammadef}
\Gamma(t) = r^{\alpha}(t) \dot r^2(t), \qquad t > T_1,
\end{equation}
and we observe that, in view of \eqref{vinf}, \eqref{eqenergia} and \eqref{momentopolare},
$$
\Gamma(t) = \frac{2m}{\alpha} + r^\alpha(t) W(x(t)) - \frac{\vert A(t) \vert^2}{r^{2-\alpha}(t)}.
$$
Using the fact that $r(t) \to \infty$ and \eqref{stimaW} and \eqref{alimitato}, we thus obtain
$$
\lim_{t \to +\infty}\Gamma(t) = \frac{2m}{\alpha}.
$$
An application of de l'Hopital's rule (compare with the proof of \cite[Theorem 7.7]{BarTerVer14}) then yields \eqref{rasy}.

As for $s$, we observe that, as a consequence of \eqref{momentopolare}, \eqref{rasy} and \eqref{alimitato},
there exist $a_0>0$ and $T_2 > T_1$ such that
$$
\vert \dot s(t) \vert \leq \frac{a_0}{t^{4/(\alpha+2)}}, \quad \mbox{ for every } t > T_2. 
$$
Therefore $\vert \dot s \vert$ is integrable at infinity, implying that $s$ admits a limit for $t \to +\infty$.
\end{proof}

We conclude this section with a further technical estimate, which will play an important role in the proof of Proposition \ref{approssimazione}.

\begin{lemma}\label{tecnico}
Let $x: [t_1,t_2] \to \mathbb{R}^3$ be a parabolic solution of \eqref{main} satisfying \eqref{largep} and assume that 
$r$ is strictly increasing on the whole $[t_1,t_2]$. Then, for any $\tau_1,\tau_2$ with $t_1 < \tau_1 \leq \tau_2 \leq t_2$,
\begin{equation}\label{fe}
\int_{\tau_2}^{t_2} \vert \dot s(t) \vert\,dt \leq  
\frac{C_1 \,r(\tau_1)^{1-\alpha/2} +\frac{C_2}{r(\tau_1)^{1+\alpha/2} - r(t_1)^{1+\alpha/2}}}
{C_3 \frac{r(\tau_1)^{1+\alpha/2} - r(t_1)^{1+\alpha/2}}{r(\tau_1)^{3(\alpha+2)/4}}}
 \frac{C_4}{r(\tau_2)^{(2-\alpha)/4}},
\end{equation}
where $C_j$ ($j=1,\ldots,4$) are positive constants depending only on $\alpha,m$ and $C_+$.
\end{lemma}

Clearly, a symmetric result can be given when $r$ is strictly decreasing on $[t_1,t_2]$.

\begin{proof}
At first, we observe that, using \eqref{stimajacobi} and Corollary \ref{cortempi}, we find
\begin{align}\label{passo1}
\dot I(\tau_1) &\geq \frac{(2-\alpha)m}{2\alpha} \int_{t_1}^{\tau_1} r^{-\alpha}(s)\,ds \geq
\frac{(2-\alpha)m}{2\alpha} \frac{\tau_1-t_1}{r^\alpha(\tau_1)} \nonumber \\
& \geq 
\frac{(2-\alpha)m}{2\alpha} \frac{1}{(1+\alpha/2) \sqrt{2C_+}}\frac{1}{r^\alpha(\tau_1)}(r(\tau_1)^{1+\alpha/2}-r(t_1)^{1+\alpha/2}). 
\end{align}
Next, for any $t \geq \tau_1$ we define
$$
X(t) = I(t)^{\frac{\alpha-2}{8}} \,\dot I(t)
$$
and we claim that $X$ is increasing on $[\tau_1,t_2]$. Indeed, observing that
$$
\dot I^2(t) = 2 I(t) \left( \vert \dot x(t) \vert^2 - r^2(t) \vert \dot s(t) \vert^2 \right) = 4 I(t) V(x(t)) - 4 I^2(t) \vert \dot s(t) \vert^2, 
$$
and recalling \eqref{lj}, we find, for every $t \in [\tau_1,t_2]$,
\begin{align*}
\dot X(t) & = \frac{\alpha-2}{8} I(t)^{\frac{\alpha-2}{8}-1} \dot I^2(t) + I(t)^{\frac{\alpha-2}{8}} \ddot I(t) \\
& = I(t)^{\frac{\alpha-2}{8}} \left(\left(1+\frac{\alpha}{2}\right) V(x(t)) + \nabla V(x(t)) \cdot x(t) + \left(1-\frac{\alpha}{2}\right) I(t) \vert \dot s(t)\vert^2 \right) \\
& \geq I(t)^{\frac{\alpha-2}{8}} \left(\left(1-\frac{\alpha}{2}\right) \frac{m}{\alpha r(t)} + \left(1+\frac{\alpha}{2}\right) W(x(t)) + \nabla W(x(t)) \cdot x(t)\right) > 0,
\end{align*}
in view of \eqref{stimaW2}. As a consequence, recalling \eqref{passo1},
\begin{align*}
I(t)^{\frac{\alpha-2}{8}} \,\dot I(t) & \geq I(\tau_1)^{\frac{\alpha-2}{8}} \,\dot I(\tau_1) \\
& \geq \frac{C_3}{r(\tau_1)^{\frac{3\alpha+2}{4}}} \left( r(\tau_1)^{1+\alpha/2}-r(t_1)^{1+\alpha/2} \right), \quad \mbox{ for every } t \in [\tau_1,t_2],
\end{align*}
where 
$$
C_3 = \frac{(2-\alpha)m}{2^{(14-\alpha)/8}\alpha} \frac{1}{(1+\alpha/2) \sqrt{C_+}}.
$$
Summing up,
\begin{equation}\label{passo2}
\dot I(t) \geq \frac{C_3}{r(\tau_1)^{\frac{3\alpha+2}{4}}} \left( r(\tau_1)^{1+\alpha/2}-r(t_1)^{1+\alpha/2} \right)I(t)^{\frac{2-\alpha}{8}}, \quad \mbox{ for every } t \in [\tau_1,t_2].  
\end{equation}
Now, we write \eqref{stimaapunto} as
$$
\vert \dot A(t) \vert \leq \frac{1}{2^{\frac{\alpha+2}{2}}} \frac{C_+}{I(t)^{1+\alpha/2}}, \quad \mbox{ for every } t \in [\tau_1,t_2].
$$
Recalling \eqref{passo1} and the fact that $\dot I$ is increasing, we find, for $t \in [\tau_1,t_2]$,
$$
\vert \dot A(t) \vert \leq \frac{\alpha \left( 1 +\frac{\alpha}{2} \right) C_+ \sqrt{2C_+}}{2^{\alpha/2}(2-\alpha)m} \frac{r^\alpha(\tau_1)}{r(\tau_1)^{1+\alpha/2}-r(t_1)^{1+\alpha/2}}
\frac{\dot I(t)}{I(t)^{1+\alpha/2}},
$$
so that
$$
\int_{\tau_1}^{t} \vert \dot A(s) \vert \,ds \leq \frac{C_2}{r(\tau_1)^{1+\alpha/2}-r(t_1)^{1+\alpha/2}}, \quad \mbox{ for every } t \in [\tau_1,t_2],
$$
where
$$
C_2 = \frac{2 \left( 1 +\frac{\alpha}{2} \right) C_+ \sqrt{2C_+}}{(2-\alpha)m}.
$$
Therefore, using the energy relation and \eqref{stimaV}, for every $t \in [\tau_1,t_2]$,
\begin{align*}
\vert A(t) \vert & \leq \vert A(\tau_1) \vert + \frac{C_2}{r(\tau_1)^{1+\alpha/2}-r(t_1)^{1+\alpha/2}} \\
& \leq  C_1 r(\tau_1)^{1-\alpha/2} + \frac{C_2}{r(\tau_1)^{1+\alpha/2}-r(t_1)^{1+\alpha/2}},
\end{align*}
where $C_1 = \sqrt{2C_+}$
Recalling \eqref{momentopolare}, we thus find
$$
\vert \dot s(t) \vert \leq \left( C_1 r(\tau_1)^{1-\alpha/2} + \frac{C_2}{r(\tau_1)^{1+\alpha/2}-r(t_1)^{1+\alpha/2}}\right) \frac{1}{2 I(t)}, \quad \mbox{ for every } t \in [\tau_1,t_2].
$$
Combining this estimate with \eqref{passo2}, we obtain, for $\tau_2 \in [\tau_1,t_2]$,
\begin{align*}
\int_{\tau_2}^{t_2} \vert \dot s(t) \vert \,dt & \leq \frac{C_1 r(\tau_1)^{1-\alpha/2} + \frac{C_2}{r(\tau_1)^{1+\alpha/2} - r(t_1)^{1+\alpha/2}}}
{C_3 \frac{r(\tau_1)^{1+\alpha/2} - r(t_1)^{1+\alpha/2}}{r(\tau_1)^{3(\alpha+2)/4}}} \frac{1}{2} \int_{\tau_2}^{t_2} \frac{\dot I(t)}{I(t)^{1+(2-\alpha)/8}}\,dt \\
& \leq \frac{C_1 r(\tau_1)^{1-\alpha/2} + \frac{C_2}{r(\tau_1)^{1+\alpha/2} - r(t_1)^{1+\alpha/2}}}
{C_3 \frac{r(\tau_1)^{1+\alpha/2} - r(t_1)^{1+\alpha/2}}{r(\tau_1)^{3(\alpha+2)/4}}}
\frac{C_4}{r(\tau_2)^{(2-\alpha)/4}},
\end{align*}
where $C_4 = 2^{(\alpha+14)/8}/(2-\alpha)$. The proof is thus concluded.
\end{proof}

\section{The approximation argument}\label{sez3}

In this section, given $\xi^+, \xi^- \in \mathbb{S}^2$, we present a result showing how an
entire parabolic solution of \eqref{main} satisfying \eqref{asintoticast} can be obtained as limit of parabolic solutions 
defined on compact intervals, provided suitable assumptions are satisfied. Notice that in this section the hypothesis
$\xi^+ \neq \xi^-$ is not needed.

\begin{proposition}\label{approssimazione}
Let $\xi^+, \xi^- \in \mathbb{S}^2$.
Suppose that, for every large $R > 0$, there exists a parabolic solution $x_R:[-\omega_R,\omega_R] \to \mathbb{R}^3 \setminus \Sigma$
of \eqref{main} satisfying $x_R(-\omega_R) = R\xi^-$, $x_R(\omega_R) = R\xi^+$,
\begin{equation}\label{limitesup}
\limsup_{R \to +\infty} \,\min_t \vert x_R(t) \vert < +\infty
\end{equation}
and
\begin{equation}\label{limiteinf}
\liminf_{R \to +\infty} \,\min_t \vert x_R(t) - c_i \vert > 0, \qquad \forall \, i =1,\ldots,N.
\end{equation} 
Finally, in the case $\min_t \vert x_R(t) \vert < K$, suppose further that 
\begin{equation}\label{stimatempi}
\limsup_{R \to +\infty}\left( t^+_R - t^-_R \right) < +\infty,
\end{equation}
being $K > \Xi + 1$ the constant fixed in Subsection \ref{stimepot} and $t^-_R < t^+_R$ the unique instants such that 
$\vert x_R(t^\pm_R) \vert = K$ and $\vert x_R(t) \vert > K$ for $t < t_R^-$ and $t > t^+_R$.

Then, there exists a parabolic solution
$x_\infty: \mathbb{R} \to \mathbb{R}^3 \setminus \Sigma$ of \eqref{main} such that,
writing $x_\infty = r_\infty s_\infty$ as in \eqref{polar1},
$$
r_\infty(t) \sim \gamma_{\alpha,m} \, \vert t \vert^{\frac{2}{2+\alpha}}, \qquad t \to \pm\infty,
$$
with $\gamma_{\alpha,m}$ given by \eqref{gamdef}, and 
$$
\lim_{t \to \pm \infty} s_{\infty}(t) = \xi^{\pm}.
$$
\end{proposition}

\begin{proof}
As a preliminary step, we notice that, for any $R$ such that $\min_t \vert x_R(t) \vert \geq K$ (if any), Lemma \ref{lemlj}
implies that $r_R(t):= \vert x_R(t) \vert$ has a unique minimum point $t_R$; in this case, we set $t^-_R = t^+_R = t_R$.
Hence, the instants $t^{\pm}_R$ are well-defined for any $R$. We also introduce the constants
$K_R = \min_t r_R(t)$ and $\tilde K_R = \max\{ K, K_R \}$ and we observe that assumption \eqref{limitesup} guarantees the existence of
$\tilde K \geq K$ such that $\tilde K_R \leq \tilde K$ for any (large) $R$. 
We split the proof in some steps.
\smallbreak
\noindent
\underline{Claim 1:} it holds that $\omega_R - t^+_R\to +\infty$ and $-\omega_R-t^-_R \to -\infty$.
\smallbreak
\noindent
Indeed, using Corollary \ref{cortempi} with $t_1 = t^+_R$ and $t_2 = \omega_R$, we obtain
\begin{align*}
\omega_R - t^+_R & \geq \frac{1}{\left(1+\alpha/2 \right)\sqrt{2C_+}} \left( R^{1+\alpha/2}-\tilde K_R^{1+\alpha/2}\right) \\
& \geq \frac{1}{\left(1+\alpha/2 \right)\sqrt{2C_+}} \left( R^{1+\alpha/2}-\tilde K^{1+\alpha/2}\right),
\end{align*}
whence the conclusion (for $-\omega_R - t^-_R$ the argument is the same).
\smallbreak
We now define
\begin{equation}\label{xrtraslata}
\tilde x_R(t) = x_R\left( t + \frac{t^-_R + t^+_R}{2} \right), \qquad 
t \in [\omega^-_R,\omega^+_R],
\end{equation}
where 
$$
\omega^-_R = -\omega_R - t^-_R - \Delta_R, \quad 
\omega^+_R = \omega_R - t^+_R + \Delta_R, \quad 
\Delta_R = \frac{t^+_R - t^-_R}{2}.
$$
Notice that assumption
\eqref{stimatempi} guarantees that $\Delta_R \leq \Delta$ for a suitable $\Delta > 0$ and $R$ large enough. Then, we have the following.
\smallbreak
\noindent
\underline{Claim 2:} there exists a $\mathcal{C}^2$-function $x_\infty:\mathbb{R} \to \mathbb{R}^3$ such that, for $R \to +\infty$,
$$
\tilde x_R \to x_{\infty} \quad \mbox{ in } \, \mathcal{C}^2_{\textnormal{loc}}(\mathbb{R}).
$$
\smallbreak
\noindent
Indeed, \eqref{limitesup} and \eqref{limiteinf} imply that
$$
\vert \tilde x_R(0) \vert, \quad \vert \dot{\tilde x}_R(0) \vert = \sqrt{2V (\tilde x_R(0))}, \quad 
\max_t \vert \ddot{\tilde x}_R(t) \vert = \max_t \vert \nabla V(\tilde x_R(t)) \vert 
$$
are bounded in $R$. Then, a standard compactness argument gives the conclusion. 
\smallbreak
\noindent
\underline{Claim 3:} $x_\infty: \mathbb{R} \to \mathbb{R}^3 \setminus \Sigma$ is a parabolic solution
of \eqref{main}
and 
$$
r_\infty(t) \sim \gamma_{\alpha,m}, \, \vert t \vert^{\frac{2}{2+\alpha}}\qquad t \to \pm\infty.
$$ 
\smallbreak
\noindent
Indeed, \eqref{limiteinf} guarantees that
$x_\infty$ has no collisions; hence, using Claim 2 we can pass to the limit both in the equation and in the energy relation, thus ensuring that
$x_\infty$ is a parabolic solution of \eqref{main}. 
Moreover, \eqref{stimatempi} implies that $r_\infty(t) \geq K$ for 
$\vert t \vert \geq \Delta$. Then, the asymptotic estimates for $r_\infty$ follow from Proposition \ref{asintotica}
(and the symmetric statement for $t \to -\infty$).
Notice that Proposition \ref{asintotica} also implies that $s_\infty$ admits a limit both for $t \to +\infty$ and for $t \to -\infty$,
but we do not know these limits to be $\xi^{\pm}$. This is indeed our final step.
\smallbreak
\noindent
\underline{Claim 4:} it holds that
\begin{equation}\label{limites}
\lim_{t \to \pm\infty }s_\infty(t) = \xi^{\pm}.
\end{equation}
\smallbreak
\noindent
We prove only the limit relation for $t \to +\infty$ (the other being analogous). 
As a first step, we fix a constant $\tilde C$ such that
$$
\frac{C_1 (\tilde K_R + 1)^{1-\alpha/2} + 
\frac{C_2}{(\tilde K_R + 1)^{1+\alpha/2} - \tilde K_R^{1+\alpha/2}}}
{C_3 \frac{(\tilde K_R + 1)^{1+\alpha/2} - \tilde K_R^{1+\alpha/2}}{(\tilde K_R + 1)^{3(\alpha+2)/4}}}
C_4 < \tilde C
$$
for any $R$ (where the constants $C_1,C_2,C_3$ and $C_4$ are the ones in Lemma \ref{tecnico}); this is possible since
$K \leq \tilde K_R \leq \tilde K$. Next, for any $\epsilon > 0$ let us take $Z_\epsilon > \tilde K + 1$ such that 
$$
\frac{\tilde C}{Z_\epsilon^{(2-\alpha)/4}} < \frac{\epsilon}{2},
$$
so that
$$
\frac{C_1 (\tilde K_R + 1)^{1-\alpha/2} + \frac{C_2}{(\tilde K_R + 1)^{1+\alpha/2} - \tilde K_R^{1+\alpha/2}}}
{C_3 \frac{(\tilde K_R + 1)^{1+\alpha/2} - \tilde K_R^{1+\alpha/2}}{(\tilde K_R + 1)^{3(\alpha+2)/4}}}
\frac{C_4}{Z_\epsilon^{(2-\alpha)/4}} < \frac{\epsilon}{2},
$$
for any $R > Z_\epsilon$. Let $\tilde t_{\epsilon,R} > \Delta'_R  > \Delta_R$ be the unique instants such that
$\tilde r_R(\tilde t_{\epsilon,R}) = Z_\epsilon$ and $\tilde r_R(\Delta'_R) = \tilde K_R + 1$ respectively, where we have employed the
usual notation $\tilde x_R = \tilde r_R \tilde s_R$.  
From Lemma \ref{tecnico} with the choices $t_1 = \Delta_R$, $\tau_1 = \Delta_R'$, $\tau_2 = \tilde t_{\epsilon,R}$ and
$t_2 = \omega_R^+$, we have that
$$
\int_{\tilde t_{\epsilon,R}}^{\omega_R^+} \vert \dot{\tilde s}_R(t) \vert \,dt < \frac{\epsilon}{2}.
$$  
On the other hand, \eqref{fa} gives
$$
\tilde t_{\epsilon,R} - \Delta_R \leq \sqrt{\frac{2\alpha}{(2-\alpha)m}} Z_\epsilon^{1+\alpha/2};
$$
hence, recalling that $\Delta_R \leq \Delta$, 
$$
\tilde t_\epsilon := \sup_{R > Z_\epsilon} \tilde t_{\epsilon,R} < +\infty.
$$
We are now in position to conclude. Indeed, for any $t > \tilde t_\epsilon$ let us take $R > Z_\epsilon$ such that
$\vert \tilde s_R(t) - s_\infty(t) \vert < \epsilon/2$ (following from the $\mathcal{C}^2_\textnormal{loc}$ convergence).
Then
\begin{align*}
\vert s_\infty(t) - \xi^+ \vert & \leq \vert s_\infty(t) - \tilde s_{R}(t) \vert + \vert \tilde s_{R}(t) - \tilde s_R(\omega^+_R) \vert 
\\ & < \frac{\epsilon}{2} + \int_{t}^{\omega_R^+} \vert \dot{\tilde s}_R(t) \vert \,dt \leq
\frac{\epsilon}{2} + \int_{\tilde t_{\epsilon,R}}^{\omega_R^+} \vert \dot{\tilde s}_R(t) \vert <  \frac{\epsilon}{2} + \frac{\epsilon}{2} = \epsilon, 
\end{align*}
thus proving \eqref{limites}.
\end{proof}

\section{The fixed-endpoints problem}\label{s4}

In view of Proposition \ref{approssimazione}, in this section we look for parabolic solutions of the (free-time) fixed-endpoints problem
\begin{equation}\label{bol}
\left\{
\begin{array}{l}
\vspace{0.1cm}
\displaystyle{\ddot x_R = \nabla V(x_R)}\\
x_R(\pm \omega_R) = R \xi^{\pm},
\end{array}
\right.
\end{equation}
with $V$ defined in \eqref{defv}. Henceforth, 
we use the notation
$$
\mathcal{A}_{[a,b]}(x) = \int_a^b \left( \frac{1}{2} \vert \dot x(t) \vert^2 + V(x(t)) \right)\,dt
$$
for any $x \in H^1([a,b]; \mathbb{R}^3)$. As well known, if $\bar{x}: [a,b] \to \mathbb{R}^3 \setminus \Sigma$ is a (non-collision) solution
of $\ddot x = \nabla V(x)$, then $\bar{x}$ is a critical point of the functional $\mathcal{A}_{[a,b]}$ on the domain
$\{ x \in H^1([a,b];\mathbb{R}^3 \setminus \Sigma) \, : \, x(a) = \bar{x}(a),\, x(b) = \bar{x}(b) \}$.
\smallbreak
We have the following result, which can be considered of independent interest. 
 
\begin{theorem}\label{thfix}
Let $K > \Xi +1$ be the constant given in Subsection \ref{stimepot}. Then, for any $R > K$ and for any $\xi^+$, $\xi^- \in \mathbb{S}^2$ with $\xi^+ \neq \xi^-$, 
there exist $\omega_R > 0$ and a parabolic solution of \eqref{bol} satisfying 
\begin{equation}\label{morsefix}
\textnormal{j}(d^2 \mathcal{A}_{[-\omega_R,\omega_R]}(x_R)) \leq 1,
\end{equation}
and
\begin{equation}\label{stimalivello}
\left(\sqrt{\frac{2m}{\alpha}}\frac{4}{2-\alpha}\right) R^{1-\alpha/2} - M \leq \mathcal{A}_{[-\omega_R,\omega_R]}(x_R) \leq 
\left(\sqrt{\frac{2m}{\alpha}}\frac{4}{2-\alpha}\right) R^{1-\alpha/2} + M,
\end{equation}
where $M > 0$ is a suitable constant not depending on $R$.
\end{theorem}

A comment about this result is in order. The existence of a parabolic solution of the fixed-endpoints problem is far from being surprising, since it
could be proved by using quite standard minimization arguments (together with Marchal's principle \cite{Mar02}). In this way, a solution having zero Morse index 
(cf. \eqref{morsefix}) can be obtained. Unfortunately, this solution is not robust when the fixed ends are sent to infinity.
 The crucial point in Theorem \ref{thfix} is the asymptotic level estimate \eqref{stimalivello}, which indeed
does not hold for minimizing parabolic solutions. This estimate, together with the Morse index bound \eqref{morsefix}, will allow us to pass to the limit as the endpoints tend to infinity along 
the fixed directions. Indeed it enables us to prove 
\eqref{limitesup}, \eqref{limiteinf} and \eqref{stimatempi} (see Section \ref{s5}); via Proposition \ref{approssimazione}, 
an entire parabolic solution of \eqref{main} 
with prescribed asymptotic directions will be therefore obtained.
\smallbreak 
The proof of Theorem \ref{thfix} will be given in four main steps. 

At first (see Section \ref{sezminmax}), we use a variational argument of min-max type to prove the existence of a parabolic solution for a modified equation of the form $\ddot x = \nabla V_\beta(x)$, with $\beta \in (0,1]$ and $V_\beta$ a potential satisfying a strong-force condition near each center. The min-max argument is similar to the one introduced in \cite{BahRab89,Gre88} dealing with the fixed-time (periodic) problem; here we look for fixed-energy solutions, therefore using the Maupertuis functional
(as in \cite{FelTan00,Tan93}). 

As a second step (see Section \ref{secgen}), we pass to the limit for $\beta \to 0^+$ so as to find the existence of a \emph{generalized} (parabolic) solution of 
$\ddot x = \nabla V(x)$ (cf. \cite{BahRab89, Tan93} again).

In the third step (see Section \ref{noncoll}), we prove that generalized solutions are actually classical ones, by showing that collisions with the set of the centers cannot occur.
To this end, we take advantage of a blow-up argument introduced in \cite{Tan93b} and highlighting the relation between 
the Morse index of the solution and the number of its collisions. This is enough to obtain the conclusion when $\alpha > 1$, while 
further information coming from regularizations techniques \cite{Spe69} is needed when $\alpha = 1$.

Finally (see Section \ref{seclivello}), we prove the Morse index formula \eqref{morsefix} as well as the level estimate \eqref{stimalivello}.

\smallbreak
The arguments in the first two steps, as well as in the third step for $\alpha > 1$, are valid more in general for parabolic solutions
joining two points $q^+, q^- \in \mathbb{R}^3 \setminus \Sigma$ with $q^+ \neq q^-$. For this reason, and not to overload the notation
emphasizing an inessential dependence on $R$, we will give the corresponding proofs in this setting. 

\subsection{A min-max argument}\label{sezminmax}

Let us first define the modified potential $V_\beta$, for $\beta \in [0,1]$, by setting
$$
V_\beta(x) = V(x) + \beta U(x), \qquad x \in \mathbb{R}^3 \setminus \Sigma,
$$
where $U \in \mathcal{C}^2(\mathbb{R}^3 \setminus \Sigma)$ is defined as
$$
U(x) = \sum_{i=1}^N \frac{m_i}{2 \vert x - c_i \vert^2} \Psi(\vert x - c_i \vert^2), 
$$
with $\Psi \in \mathcal{C}^2(\mathbb{R}^+;[0,1])$ a cut-off function such that $\Psi(r) = 1$ if $0 \leq r \leq \delta^*$ and $\Psi(r) = 0$ if $r \geq 2\delta^*$.
%is chosen in such a way that the following properties hold:
%\begin{itemize}
%\item[(U1)] $U(x) \geq 0$ for any $x \in \mathbb{R}^3 \setminus \Sigma$,
%\item[(U2)] $U(x) = 0$ for $\vert x \vert \geq 1$,
%\item[(U3)] $\displaystyle{U(x) \geq \frac{1}{\vert x - c_i \vert^2}}$ for $\vert x - c_i \vert \leq \delta^*$, with $\delta^* > 0$ as in
%\eqref{deltastar}.
%\end{itemize}
At this point, we can introduce the Maupertuis functional
$$
\mathcal{M}_\beta(u) = \int_{-1}^1 \vert \dot u(t) \vert^2 \,dt \int_{-1}^1 V_\beta(u(t))\,dt
$$
defined on the Hilbert manifold
$$
\Gamma = \Gamma_{q^{\pm}} = \Big\{ u \in H^1([-1,1];\mathbb{R}^3 \setminus \Sigma) \, : \, u( \pm 1) = q^{\pm} \Big\}. 
$$
As well-known (see, for instance, \cite[Theorem 4.1]{AmbCot93} and \cite[Appendix B]{SoaTer13}) $\mathcal{M}_\beta$ is smooth 
and any critical point $u_\beta \in \Gamma$ 
satisfies, for $t \in [-1,1]$, 
\begin{equation}\label{eqmau}
\ddot u_\beta(t) = \omega_\beta^2\, \nabla V_\beta (u_\beta(t)), \qquad \frac{1}{2}\vert \dot u_\beta(t) \vert^2 - \omega_\beta^2\, V_\beta(u_\beta(t)) = 0,
\end{equation}
where 
\begin{equation}\label{defomega}
\omega_\beta = \left(\frac{\int_{-1}^1 \vert \dot u_\beta \vert^2}{2 \int_{-1}^1 V_\beta(u_\beta)}\right)^{1/2}.
\end{equation}
Notice that, since $q^+ \neq q^-$, $u_\beta$ is not constant: as a consequence, $\omega_\beta > 0$ and the function
\begin{equation}\label{defxbeta}
x_\beta(t) = u_\beta\left( \frac{t}{\omega_\beta}\right), \qquad t \in [-\omega_\beta,\omega_\beta],
\end{equation}
is a parabolic solution of $\ddot x_\beta = \nabla V_\beta(x_\beta)$ on the interval $[-\omega_\beta,\omega_\beta]$ and, of course, $x_\beta(\pm \omega_\beta) = q^{\pm}$.
\smallbreak
In the next lemma we collect the compactness properties of $\mathcal{M}_\beta$ which will be used later.

\begin{lemma}\label{mauplem}
The following hold true:
\begin{itemize}
\item[(M1)] for any $\beta \geq 0$, $\mathcal{M}_\beta$ is coercive at infinity, that is, if $\Vert u_n \Vert \to +\infty$, then
$$
\lim_{n \to +\infty} \mathcal{M}_\beta(u_n) = +\infty;
$$
\item[(M2)] for any $\beta > 0$, $\mathcal{M}_\beta$ is ``coercive at the boundary'', that is, if $u \in \partial\,\Gamma$ and $u_n \rightharpoonup u$ weakly in $H^1$, then
$$
\lim_{n \to +\infty} \mathcal{M}_\beta(u_n) = +\infty;
$$ 
\item[(M3)] for any $\beta > 0$, $\mathcal{M}_\beta$ satisfies the Palais-Smale condition, that is, if $\mathcal{M}_\beta(u_n)$ is bounded and
$\nabla \mathcal{M}_\beta(u_n) \to 0$, then there exists 
$u \in \Gamma$ such that $u_n \to u$ strongly in $H^1$ (up to subsequences).
\end{itemize}
\end{lemma}

\begin{proof}
As for (M1), we argue similarly as in \cite[Lemma 3.2]{BarTerVer14}. Suppose by contradiction that $\Vert u_n \Vert \to \infty$ and
$\mathcal{M}_\beta(u_n)$ is bounded from above. 
Then $\int_{-1}^1 \vert \dot u_n \vert^2 \to +\infty$ and, therefore,
$$
\delta_n := \int_{-1}^1 V_\beta(u_n) \to 0^+.
$$
As a consequence, there exists
$t_n \in [0,1]$ such that $V_\beta(u_n(t_n)) \leq \delta_n$; hence $\vert u_n(t_n) \vert \to +\infty$.
From \eqref{stimaV}, we thus 
$$
\vert u_n(t_n) \vert \geq \left(\frac{C_-}{\delta_n}\right)^{1/\alpha}.
$$
Then, for large $n$, we have
$$
\int_{-1}^1 \vert \dot u_n \vert^2 \geq \int_{-1}^{t_n} \vert \dot u_n \vert^2 \geq \frac{1}{2} \vert u_n(t_n) - q^- \vert^2 \geq \frac{1}{2}(C_-^{1/\alpha}\delta_n^{-1/\alpha}-\vert q^- \vert)^2 \geq \frac{1}{4}C_-^{2/\alpha}\delta_n^{-2/\alpha};
$$
as a consequence
$$
\mathcal{M}_\beta(u_n) \geq \frac{1}{4}C_-^{2/\alpha}\delta_n^{1-\frac{2}{\alpha}}
$$
contradicting the fact that $\mathcal{M}_\beta(u_n)$ is bounded for above.
\smallbreak
As for (M2), we first observe that 
$$
\mathcal{M}_\beta(u_n) \geq \frac{1}{2}\big( \vert q^+ \vert - \vert q^- \vert \big)^2 \int_{-1}^1 V_\beta(u_n);
$$
hence, we only need to show that $\int_{-1}^1 V_\beta(u_n) \to \infty$. This can be proved as in 
\cite[Lemma 5.3]{AmbCot93} with obvious modifications.

Finally, we deal with (M3). Let $(u_n) \subset \Gamma$ be a Palais-Smale sequence. From (M1) we know that
$\Vert u_n \Vert$ is bounded, so that $u_n \rightharpoonup u$ weakly in $H^1$; moreover, $u \in \Gamma$ in view of (M2).
Hence, we only need to show that $u_n \to u$ strongly. To this end, we write
\begin{align*}
\langle \nabla \mathcal{M}_\beta(u_n), u_n - u \rangle = & 2 \int_{-1}^1 \vert \dot u_n \vert^2 \int_{-1}^1 V_\beta(u_n) - 
2 \int_{-1}^1( \dot u_n \cdot \dot u) \int_{-1}^1 V_\beta(u_n) + \\
& \int_{-1}^1 \vert \dot u_n \vert^2 \int_{-1}^1 \big(\nabla V_\beta(u_n) \cdot (u_n - u)\big) 
\end{align*}
Since $u_n \rightharpoonup u$ weakly in $H^1$ and $u_n \to u$ uniformly in $[-1,1]$, with $u \in \Gamma$, it holds that 
$$
2 \int_{-1}^1( \dot u_n \cdot \dot u) \int_{-1}^1 V_\beta(u_n) \to 2 \int_{-1}^1 \vert \dot u \vert^2 \int_{-1}^1 V_\beta(u)
$$
and
$$
\int_{-1}^1 \vert \dot u_n \vert^2 \int_{-1}^1 \big(\nabla V_\beta(u_n) \cdot (u_n - u)\big) \to 0.
$$
Therefore, taking into account that $\langle \nabla \mathcal{M}_\beta(u_n), u_n - u \rangle \to 0$ 
(notice that $u_n - u \in H^1_0([-1,1])$), we infer that
$$
2 \int_{-1}^1 \vert \dot u_n \vert^2 \int_{-1}^1 V_\beta(u_n) \to 2 \int_{-1}^1 \vert \dot u \vert^2 \int_{-1}^1 V_\beta(u).
$$
Since $\int_{-1}^1 V_\beta(u_n) \to \int_{-1}^1 V_\beta(u)$, we thus have $\int_{-1}^1 \vert \dot u_n \vert^2 \to \int_{-1}^1 \vert \dot u \vert^2$.
As a consequence, $u_n \to u$ strongly in $H^1$, as desired.
\end{proof}

Now we are going to describe the min-max argument. For any $h \in \mathcal{C}\left(\mathbb{S}^1,\Gamma\right)$ and for $i=1,2$, set
$$
\tilde h_i : \mathbb{S}^1 \times [-1,1] \to \mathbb{S}^2, \qquad
(s,t) \mapsto \frac{h(s)(t)-c_i}{\vert h(s)(t) - c_i \vert}.
$$
Since $\tilde h_i(s,\pm1) = q^{\pm}$ for any $s \in \mathbb{S}^1$, the map
$\tilde h_i$ can be identified with a continuous self-map on $\mathbb{S}^2$ and so it has a well-defined degree
$\textnormal{deg}_{\mathbb{S}^2}(\tilde h_i)$ \cite{GraDug03}. We can thus define the class 
\begin{equation}\label{minmaxdef}
\Lambda = \Lambda_{q^{\pm}} = \left\{ h \in \mathcal{C}\left(\mathbb{S}^1,\Gamma\right) \, : \, 
\textnormal{deg}_{\mathbb{S}^2}(\tilde h_1 ) \neq 0 = \textnormal{deg}_{\mathbb{S}^2}(\tilde h_2 )  \right\} 
\end{equation}
(it is clear that this set is non-empty) and the associated min-max value
\begin{equation}\label{deflivello}
c_\beta = c_{\beta,q^{\pm}} = \inf_{h \in \Lambda} \sup_{s \in \mathbb{S}^1} \mathcal{M}_\beta(h(s)).
\end{equation}
We first show that the levels $c_\beta$ are bounded and bounded away from zero.

\begin{lemma}\label{lemliv}
There exist $c_*, c^* > 0$ such that
$$
c_* \leq c_\beta \leq c^*, \quad \mbox{ for any } \beta \in [0,1].
$$
\end{lemma}

\begin{proof}
We first observe that the function $\beta \mapsto c_\beta$ is non-decreasing. As a consequence,
$c_0 \leq c_\beta \leq c_1$ for any $\beta \in [0,1]$. We thus only need to show that $c_0 > 0$.
By contradiction, assume that there exist sequences $(h_n) \subset \Lambda$ and $(s_n) \subset \mathbb{S}^1$ such that
$\mathcal{M}_0(h_n(s_n)) \to 0$. Then
$$
\int_{-1}^1 V(h_n(s_n)) \leq \frac{2 \mathcal{M}_0(h_n(s_n))}{(\vert q^+ \vert - \vert q^- \vert)^2} \to 0
$$
and the very same arguments used in the proof of (M1) in Lemma \ref{mauplem} show that 
$\mathcal{M}_0(h_n(s_n)) \to +\infty$, a contradiction.
\end{proof}

We are now in position to state and prove the main result of this subsection, ensuring the existence of 
critical points at level $c_\beta$ and having Morse index at most $1$.

\begin{proposition}\label{valcritico}
For any $\beta > 0$, $c_\beta$ is a critical value for the functional $\mathcal{M}_\beta$. 
In particular, there exists $u_{\beta} = u_{\beta,q^{\pm}} \in \Gamma$ such that
$$
\mathcal{M}_\beta(u_\beta) = c_\beta, \quad \nabla \mathcal{M}_\beta(u_\beta) = 0, \quad
\textnormal{j}\left( d^2 \mathcal{M}_\beta(u_\beta) \right) \leq 1.
$$
\end{proposition}

\begin{proof}[Sketch of the proof]
The fact that $c_\beta$ is a critical value for $\mathcal{M}_\beta$ follows from standard arguments of Critical Point Theory.
Indeed, the compactness properties of $\mathcal{M}_\beta$ collected in Lemma \ref{mauplem}
allow us to prove a Deformation Lemma on the lines of \cite[Proposition 1.17]{BahRab89}
or \cite[Proposition 1.6]{Tan93}. Then, a well-known min-max principle (cf. \cite[Theorem 4.2]{Str08}) yields the conclusion.

To prove that $\mathcal{M}_\beta^{-1}(c_\beta)$ contains a critical point $u_\beta$ with 
$\textnormal{j}\left( d^2 \mathcal{M}_\beta(u_\beta) \right) \leq 1$, one has to argue similarly as in the proof
of \cite[Proposition 1.5 (iii)]{Tan93}. The crucial point here is that the Morse index cannot exceed the value $1$ since,
in the definition of $c_\beta$, the one-dimensional manifold $\mathbb{S}^1$ is involved (see also \cite{Tan93b}). 
\end{proof}

\begin{remark}
We observe that the results in this section could be proved also using different min-max classes, as for instance
$\Lambda' = \left\{ h \in \mathcal{C}\left(\mathbb{S}^1,\Gamma\right) :  
\textnormal{deg}_{\mathbb{S}^2}(\tilde h_1 ) \neq 0  \right\}$. The reason for the choice of the class $\Lambda$
is that (as a direct consequence of the homotopy invariance of the degree) for any $h \in \Lambda$ there exists $s_h \in \mathbb{S}^1$ such that
\begin{equation}\label{segmento}
h(s_h)([-1,1]) \cap [c_1,c_2] \neq \emptyset,
\end{equation}  
where $[c_1,c_2] = \{ \lambda c_1 + (1-\lambda) c_2 : \lambda \in [0,1]\}$ is the segment joining $c_1$ and $c_2$.
This property, which will play a crucial role in our next arguments 
(see the final part of Section \ref{seclivello}), does not hold for min-max classes like $\Lambda'$.
\end{remark}

\subsection{Generalized solutions}\label{secgen}

Our goal now is to study the convergence for $\beta \to 0^+$ of the functions $u_\beta \in \Gamma$ given in Proposition \ref{valcritico}.
To this end, we state and prove the following lemma.

\begin{lemma}\label{limitebeta}
There exist $M_* > 0$ and $\omega_*, \omega^*$ with $\omega_*, \omega^* > 0$ such that, for any $\beta \in (0,1]$,
$$
\int_{-1}^1 \vert \dot u_\beta \vert^2 \leq M_* \quad \mbox{ and } \quad
\omega_* \leq \omega_\beta \leq \omega^*
$$ 
where $\omega_\beta$ is defined in \eqref{defomega}.
\end{lemma}

\begin{proof}
The fact that $\int_{-1}^1 \vert \dot u_\beta \vert^2$ is bounded follows immediately from (M1) of Lemma \ref{mauplem} together with Lemma \ref{lemliv}: indeed,
if $\int_{-1}^1 \vert \dot u_\beta \vert^2 \to +\infty$ then $\Vert u_\beta \Vert \to +\infty$, so that
$$
c^* \geq c_\beta = \mathcal{M}_\beta(u_\beta) \geq \mathcal{M}_0(u_\beta) \to +\infty,
$$
a contradiction. As a consequence $\Vert u_\beta \Vert$ is bounded and we easily conclude that $\omega_\beta$ is bounded from above, as well.
Finally, from Lemma \ref{lemliv} we have
$$
\omega_\beta = \frac{\int_{-1}^1 \vert \dot u_\beta \vert^2}{\sqrt{2 c_\beta}} \geq \frac{(\vert q^+ \vert - \vert q^- \vert)^2}{2\sqrt{2c_*}}
$$ 
so that $\omega_\beta$ is bounded away from zero. This concludes the proof.
\end{proof}

From Lemma \ref{limitebeta}, it follows that (up to subsequences) $\omega_\beta \to \omega_0 \in [\omega_*,\omega^*]$ and $u_\beta \rightharpoonup u_0$ weakly in $H^1$. Moreover
the set
\begin{equation}\label{defd0}
D_0 = u_0^{-1}(\Sigma)
\end{equation}
has zero measure; indeed, by Fatou's lemma and Lemma \ref{lemliv}
\begin{align*}
\int_{-1}^1 V(u_0) & \leq \liminf_{\beta \to 0^+}\int_{-1}^1 V(u_\beta) \leq \liminf_{\beta \to 0^+} \int_{-1}^1 V_\beta(u_\beta) \\
& = \liminf_{\beta \to 0^+}  \frac{c_\beta}{\int_{-1}^1 \vert \dot u_\beta \vert^2} \leq \frac{2c^*}{ (\vert q^+ \vert - \vert q^- \vert)^2}.
\end{align*}
Then, arguing as in \cite[p. 374]{Tan93}, we can prove that the function
\begin{equation}\label{defx0}
x_0(t) = u_0\left( \frac{t}{\omega_0}\right), \qquad t \in [-\omega_0,\omega_0],
\end{equation}
is a generalized parabolic solution of $\ddot x = \nabla V(x)$, that is:
\begin{itemize}
\item[i)] $x_0 \in \mathcal{C}([-\omega_0,\omega_0];\mathbb{R}^3)$ and $x_0(\pm \omega_0) = q^{\pm}$,
\item[ii)] the set $E_0 = x_0^{-1}(\Sigma) = \omega_0 \,D_0$ has zero measure,
\item[iii)] $x_0 \in \mathcal{C}^2([-\omega_0,\omega_0] \setminus E_0;\mathbb{R}^3 \setminus \Sigma)$ and, for any 
$t \in [-\omega_0,\omega_0] \setminus E_0$,
$$
\ddot x_0(t) = \nabla V(x_0(t)), \qquad \frac{1}{2}\vert \dot x_0(t) \vert^2 - V(x_0(t)) = 0.
$$
\end{itemize}
Of course, such a solution is actually a classical one whenever $D_0 = \emptyset$.

\begin{remark}\label{remlivellozero}
For further convenience, we also observe that, if $D_0 = \emptyset$, then $u_\beta \to u_0$ in $\mathcal{C}^2$, $u_0$ is
is a critical point of $\mathcal{M}_0$
and, moreover,
\begin{equation}\label{livellozero}
\mathcal{M}_0(u_0) = c_0,
\end{equation}
where $c_0$ is defined in \eqref{deflivello} (for $\beta = 0$). To prove \eqref{livellozero}, we first observe that
(as a consequence of the $\mathcal{C}^2$-convergence) 
$c_\beta = \mathcal{M}_\beta(u_\beta) \to \mathcal{M}_0(u_0)$; moreover, we have already noticed (see the proof of Lemma \ref{lemliv})
that $\beta \mapsto c_\beta$ is non-decreasing. Hence, $c_0 \leq \mathcal{M}_0(u_0)$. Now, assume by contraction that
$c_0 < \mathcal{M}_0(u_0)$; then, there exists $h \in \mathcal{C}(\mathbb{S}^1,\Gamma)$
such that, for any $\beta \in (0,1]$,
$$
\sup_{s \in \mathbb{S}^1} \mathcal{M}_0(h(s)) < \mathcal{M}_0(u_0) \leq \mathcal{M}_\beta(u_\beta) = c_\beta \leq \sup_{s \in \mathbb{S}^1} \mathcal{M}_\beta(h(s)). 
$$
On the other hand,
$$
\vert \mathcal{M}_\beta(h(s)) - \mathcal{M}_0(h(s)) \vert \leq \beta \int_{-1}^1 \left\vert \frac{d}{dt} h(s) \right\vert^2 \int_{-1}^1 U(h(s)) \to 0, 
$$
uniformly in $s \in \mathbb{S}^1$ for $\beta \to 0^+$, a contradiction.
\end{remark}

\subsection{Non-collision solutions}\label{noncoll}

In this section we show that $D_0 = \emptyset$. To this end, we assume by contradiction that $D_0 \neq \emptyset$ and we define
\begin{equation}\label{defnu}
\nu = \# D_0 > 0.
\end{equation}
We also set, for $\alpha \in [1,2)$, 
\begin{equation}\label{ialpha}
i(\alpha) = \max \left\{ k \in \mathbb{N} \, : \, k < \frac{2}{2-\alpha}\right\}.
\end{equation}
For comments about the meaning of this definition, we refer to \cite[Section 4]{Tan93b}.
Here we simply notice that $i(1) = 1$ and $i(\alpha) > 1$ for $\alpha \in (1,2)$.

The next proposition is analogous to \cite[Proposition 4.1]{Tan93}.

\begin{lemma}\label{morsetanaka}
It holds that
$$
\liminf_{\beta \to 0^+}\, \textnormal{j}(d^2 \mathcal{M}_\beta(u_\beta)) \geq i(\alpha) \nu. 
$$
\end{lemma}

\begin{proof}[Sketch of the proof]
The proof follows the same lines of the one of \cite[Proposition 4.1]{Tan93}, investigating the asymptotic behavior of the 
Morse indexes $\textnormal{j}(d^2 \mathcal{M}_\beta(u_\beta))$ for $\beta \to 0^+$ via a blow-up argument.
The minor difference here comes from the proof of the convergence of the blow-up sequence and, 
for the reader's convenience, 
we sketch some details (similar arguments will also appear in the subsequent sections). 

Let $\tau_0 \in D_0 \subset (-1,1)$ and assume, to fix the ideas,
$u_0(\tau_0) = c_1$. Then, using the fact that $D_0$ has zero measure,
it is possible to find $\tau^-_\beta, \tau_\beta, 
\tau^+_\beta \in (-1,1)$ such that 
$\tau^-_\beta < \tau_\beta < \tau^+_\beta$, $\delta_\beta := \vert u_\beta(\tau_\beta) - c_1 \vert = \min_t \vert u_\beta(t) - c_1 \vert \to 0^+$,
$$
\vert u_\beta(\tau^\pm_\beta) - c_1 \vert = \delta^*
\quad \mbox{ and }
\quad \vert u_\beta(t) - c_1 \vert \leq \delta^*, \quad \mbox{ for any } t \in [\tau^-_\beta,\tau^+_\beta].
$$   
Since $u_\beta \to u_0$ uniformly, both $\tau_\beta - \tau^-_\beta$ and $\tau^+_\beta - \tau_\beta$ are bounded away from zero.
Let us define
$$
d= \lim_{\beta \to 0^+}\frac{\beta}{\delta_\beta^{2-\alpha}};
$$
we give the details only in the case $d < +\infty$ (for $d = +\infty$, see \cite{Tan93}). Let us consider $x_\beta$ as defined in \eqref{defxbeta}
and set
$$
v_\beta(t) = \frac{1}{\delta_\beta}\left(x_\beta\left( \delta_\beta^{1+\alpha/2}t + \tau_\beta\, \omega_\beta \right) - c_1 \right), \qquad t \in [-\gamma_\beta,\sigma_\beta],
$$
where
$$
-\gamma_\beta = \frac{\left( \tau^-_\beta - \tau_\beta \right)\omega_\beta }{\delta_\beta ^{1+\alpha/2}} \quad 
\mbox{and} \quad  
\sigma_\beta = \frac{\left( \tau^+_\beta - \tau_\beta \right)\omega_\beta }{\delta_\beta ^{1+\alpha/2}}.
$$
Notice that $\vert v_\beta(0) \vert = 1$, $\vert v_\beta(t) \vert \geq 1$ and $\vert \delta_\beta v_\beta(t) + c_1 \vert \leq \delta^*$ for
$t \in  [-\gamma_\beta,\sigma_\beta]$. 
An easy computation shows that, writing $V$ as in \eqref{vsing}, $v_\beta$ satisfies
$$
\ddot v_\beta = - \frac{m_1 v_\beta}{\vert v_\beta \vert^{\alpha+2}} - \frac{\beta}{\delta_\beta^{2-\alpha}} \frac{m_1 v_\beta}{\vert v_\beta \vert^4}
+\delta_\beta^{1+\alpha}\, \nabla \Phi_1(\delta_\beta v_\beta + c_1)
$$
and
$$
\frac{1}{2}\vert \dot v_\beta \vert^2 = \frac{m_1}{\alpha \vert v_\beta \vert^{\alpha}} + \frac{1}{2} \frac{\beta}{\delta_\beta^{2-\alpha}} \frac{m_1}{\vert v_\beta \vert^2} + \delta_\beta^{\alpha} \Phi_1(\delta_\beta v_\beta + c_1).
$$
Also, recalling that $\omega_\beta$ are bounded away from zero (see Lemma \ref{limitebeta}) we have that $-\gamma_\beta \to -\infty$ and $\sigma_\beta \to +\infty$. As a consequence, it is easy to see
that $v_\beta \to v_0$ in $\mathcal{C}^2_{\textnormal{loc}}(\mathbb{R})$, where $v_0$ satisfies 
$$
\ddot v_0 = - \frac{m_1 v_0}{\vert v_0 \vert^{\alpha+2}} - d \, \frac{m_1 v_0}{\vert v_0 \vert^4}
$$ 
and
$$
\frac{1}{2}\vert \dot v_0 \vert^2 = \frac{m_1}{\alpha \vert v_0 \vert^{\alpha}} + \frac{d}{2} \frac{m_1}{\vert v_0 \vert^2}.
$$
From now on, the proof follows exactly the one in \cite{Tan93}.	
\end{proof}

In view of the above lemma, and recalling that $\textnormal{j}(d^2 \mathcal{M}_\beta(u_\beta)) \leq 1$
(see Proposition \ref{valcritico}) we immediately see that $\nu = 0$ whenever $\alpha > 1$, contradicting \eqref{defnu}.
Hence, the proof that $D_0$ is empty is concluded in this case.
\smallbreak
If $\alpha = 1$, Lemma \ref{morsetanaka} (again combined with Proposition \ref{valcritico}) gives $\nu = 1$
and an additional argument is needed, requiring $\vert q^+ \vert = \vert q^- \vert > K$.
Let $t_0$ be the (unique) instant such that
$x_0(t_0) \in \Sigma$ and, to fix the ideas, assume that $x_0(t_0) = c_1$. On one hand, arguing as in the proof of Theorem
\cite[Theorem 0.1]{Tan93b}, we can see that the limit
\begin{equation}\label{limitetanaka}
\lim_{t \to t_0} \frac{x_0(t) - c_1}{\vert x_0(t) - c_1 \vert} = \xi_0 \in \mathbb{S}^2
\end{equation}
exists (that is, both the limits for $t \to t_0^{\pm}$ exist and they are equal).
On the other hand, we can regularize the equation as described in \cite{Spe69}.
More precisely, we set  
$$
\tau(t) = \int_{t_0}^t \frac{d\zeta}{\vert x_0(\zeta) - c_1 \vert}, \qquad t \in [-\omega_0,\omega_0],
$$
and we denote by $t(\tau)$ its inverse function, defined on the interval
$[\tau^-,\tau^+]$ with $\tau^{\pm} = \int_{t_0}^{\pm \omega_0}  d\zeta/\vert x_0(\zeta) - c_1 \vert$; moreover, for any $\tau \neq 0$, let
$$
\begin{array}{l}
\vspace{0.1cm}
x(\tau) = x_0(t(\tau)) - c_1 \\
\vspace{0.1cm}
\displaystyle{y(\tau) = \frac{d}{d\tau}x(\tau)} \\
\displaystyle{w(\tau) = \frac{1}{\vert x(\tau) \vert} \left[ \left( \frac{d}{d\tau} \vert x(\tau) \vert\right)  y(\tau) - m_1 
x(\tau) \right]}.
\end{array}
$$  
Then, the function $z(\tau) = (x(\tau),y(\tau),w(\tau))$ satisfies the differential equation
$$
z'(\tau) = F(z(\tau)), \qquad \tau \neq 0
$$
where $F = (F_1,F_2,F_3) \in \mathcal{C}^\infty((\mathbb{R}^3 \setminus \Sigma') \times \mathbb{R}^3 \times \mathbb{R}^3)$ with
$\Sigma' = \{ c_2 - c_1, \ldots, c_N - c_1\}$,
$$
\begin{array}{l}
\vspace{0.2cm}
F_1(z) = y \\
\vspace{0.2cm}
\displaystyle{F_2(z) = w + \vert x \vert^2 \nabla \Phi_1(x + c_1)} \\
\displaystyle{F_3(z) = (x \cdot y) \nabla \Phi_1(x + c_1) + \big( 2 \Phi_1(x+c_1) + x \cdot \nabla \Phi_1(x+c_1)\big) y} 
\end{array}
$$
and $\Phi_1$ given in \eqref{vsing}. Using the estimates in \cite[Section 7]{Spe69}, it follows that
the limit $z_0:=\lim_{\tau \to 0} z(\tau)$ exists
with
$$
z_0 = (0,0,c_{m_1} \xi_0),
$$
where $c_{m_1}$ is a suitable non-zero constant depending only on $m_1$
and $\xi_0$ is as in \eqref{limitetanaka}. Hence, $z$ satisfies the Cauchy problem
$$
z' = F(z), \qquad z(0) = z_0
$$
for any $\tau \in [\tau^-,\tau^+]$. 
Since $F$ fulfills
$$
\begin{array}{l}
\vspace{0.2cm}
F_1(x,-y,w) = -F_1(x,y,w) \\
\vspace{0.2cm}
F_2(x, -y,w) = F_2(x,y,w) \\
F_3(x,-y,w) = -F_3(x,-y,w) 
\end{array}
$$
for any $(x,y,w) \in (\mathbb{R}^3 \setminus \Sigma') \times \mathbb{R}^3 \times \mathbb{R}^3$,
it is immediate to see that $x$ satisfies
$$
x(\tau) = x(-\tau), \quad \mbox{ for} \; \vert\tau\vert \leq \min\{-\tau^-,\tau^+\}.
$$
If $-\tau^- = \tau^+$, this is impossible whenever $q^- \neq q^+$. Hence, we can assume that
$-\tau^- \neq \tau^+$ and, to fix the ideas, that $-\tau^- < \tau^+$; then $t_0 < 0$ and
$$
x_0(t) = x_0(2t_0-t), \quad \mbox{ for every } t \in [-t_0-\omega_0, t_0 + \omega_0]. 
$$
In particular
$$
\vert x_0(-\omega) \vert = \vert q^- \vert = \vert x_0(2t_0 + \omega) \vert.
$$
Since $\vert q^- \vert > K$ is large enough, Lemma \ref{lemlj} implies that, 
defining $r_0(t) = \vert x_0(t) \vert$, it holds $\dot r_0(-\omega_0) < 0$.
Therefore, $\dot r_0(2t_0 + \omega_0) > 0$ so that, again in view of Lemma \ref{lemlj},
$r_0(t) > \vert q^- \vert$ for any $t \in (2t_0+\omega_0,\omega_0]$, contradicting the fact that
$r_0(\omega_0) = \vert q^+ \vert = \vert q^- \vert$.

\subsection{Morse index and level estimates}\label{seclivello}

From now, due to our assumption $q^{\pm} = R \xi^{\pm}$, we need to emphasize the dependence on $R$ in our notation. Accordingly, 
the function $u_{\beta}$ as well as the time-interval $\omega_\beta$ appearing in Section \ref{sezminmax} will be denoted by 
$u_{\beta,R}$ and $\omega_{\beta,R}$, respectively. Also, with reference to Section \ref{secgen}, we will write
$u_R$ and $\omega_R$ for the limits of $u_{\beta,R}$ and $\omega_{\beta,R}$ as $\beta \to 0^+$, previously denoted by $u_0$ and $\omega_0$. Finally, $
x_R(t) = u_R(t/\omega_R)$ for $t \in [-\omega_R,\omega_R]$ (compare with \eqref{defx0}). 
\smallbreak
We first prove the Morse index formula \eqref{morsefix}. Since $u_{\beta,R} \to u_R$ in $\mathcal{C}^2$ (see Remark \ref{remlivellozero}), it is easy to see that
$\textnormal{j}(d^2 \mathcal{M}_0(u_R)) \leq 1$. On the other hand, a straightforward computation shows that
$$
d^2 \mathcal{M}_0(u_R)[v,v] = \left(\int_{-\omega_R}^{\omega_R} \vert \dot x_R \vert^2 \right) d^2 \mathcal{A}_{[-\omega_R,\omega_R]}(x_R)[y,y] - 4 \left( \int_{-\omega_R}^{\omega_R} \nabla V(x_R) \cdot y\right)^2,
$$
where $v \in H^1_0([-1,1])$ and $y(t) = v(t/\omega_R)$ for $t \in [-\omega_R,\omega_R]$. Hence,
$$
\textnormal{j}(d^2 \mathcal{A}_{[-\omega_R,\omega_R]}(x_R)) \leq 1,
$$
as desired.
\smallbreak
Now, we prove the estimate from above in \eqref{stimalivello}.
To this end, we first recall the notation in Section \ref{sezminmax} and we choose an arbitrary 
$\gamma \in \Lambda_{K\xi^{\pm}}$. Then, we take
$\eta^+:[1,+\infty) \to [K,+\infty)$ and $\eta^{-}: (-\infty,-1] \to [K,+\infty)$  as the solutions of the Cauchy problems
$$
\dot\eta^{\pm} = \pm\sqrt{2 V(\xi^{\pm}\eta^{\pm})}, \qquad \eta^{\pm}(\pm 1) = K
$$
and we define $\tau^+_R,\tau^-_R$ (for $R > K$) as the unique points such that $\eta^{\pm}(\tau^{\pm}_R) = R$.
As a next step, we set, for any $s \in \mathbb{S}^1$,
$$
\zeta(s)(t) = 
\begin{cases}
\xi^+ \eta^+(t) & \; \mbox{ for } t \in [1,\tau_R^+] \\
\gamma(s)(t) & \; \mbox{ for } t \in [-1,1] \\
\xi^- \eta^-(t) & \; \mbox{ for } t \in [\tau^-_R,-1]
\end{cases}
$$
and
$$
h(s)(t) = \zeta(s)\left( \tau_R^- + \frac{1}{2} (\tau^+_R - \tau^-_R) (t+1)\right), \quad
\mbox{ for any } t \in [-1,1],
$$
in such a way that $h \in \Lambda_{R\xi^{\pm}}$. We also set,
for any $T > 0$, $x_T(s)(t)=h(s)(t/T)$ for $t \in [-T,T]$.
We have, for any $s \in \mathbb{S}^1$,
\begin{align*}
\sqrt{\mathcal{M}_0(h(s))} & = \frac{1}{\sqrt{2}} \inf_{T > 0} \mathcal{A}_{[-T,T]}(x_T(s))
\leq \frac{1}{\sqrt{2}}\int_{\tau^-_R}^{\tau^+_R} 
\left( \frac{1}{2} \vert \dot \zeta(s)(t)\vert^2 + V(\zeta(s)(t))\right)\,dt \\
& \leq 
\frac{1}{\sqrt{2}}\int_{\tau^-_R}^{-1} 
\left( \frac{1}{2} \vert \dot \eta^-(t)\vert^2 + V(\xi^-\eta^-(t))\right)\,dt
+ \frac{1}{\sqrt{2}}\mathcal{A}_{[-1,1]}(\gamma(s))
\\
& \quad + 
\frac{1}{\sqrt{2}}\int_{1}^{\tau^+_R} 
\left( \frac{1}{2} \vert \dot \eta^+(t)\vert^2 + V(\xi^+\eta^+(t))\right)\,dt \\
& \leq M_+ + \int_{-1}^{\tau_R^-} \sqrt{V(\xi^-\eta^-(t))} \dot \eta^-(t) \,dt + \int_{1}^{\tau_R^+} \sqrt{V(\xi^+\eta^+(t))} \dot \eta^+(t) \, dt
\\ & = M_+ + \int_{K}^R \sqrt{V(\xi^- r)}\,dr + \int_{K}^R \sqrt{V(\xi^+ r)}\,dr,
\end{align*}
with $M_+ > 0$ a suitable constant not depending on $R$ and $s$.
Now, using \eqref{stimak} we find
$$
\sqrt{V(\xi^{\pm}r)} \leq \sqrt{\frac{m}{\alpha}} \frac{1}{r^{\alpha/2}} + \frac{C_+}{r^{2 + \alpha/2}}, \quad
\mbox{ for every } r \geq K,
$$
so that, with a simple computation,
$$
\sqrt{\mathcal{M}_0(h(s))} \leq \sqrt{\frac{m}{\alpha}}\frac{4}{2-\alpha} R^{1-\alpha/2} + M_+ + \frac{4C_+}{2+\alpha}, \quad \
\mbox{ for every } s \in \mathbb{S}^1.
$$
Recalling the definition of $c_0$ given in \eqref{deflivello}, the fact $\mathcal{M}_0(u_R) = c_0$ (compare with \eqref{livellozero}) and the well-known relation
\begin{equation}\label{azione}
\sqrt{\mathcal{M}_0(u_R)} = \frac{1}{\sqrt{2}}\mathcal{A}_{[-\omega_R,\omega_R]}(x_R),
\end{equation}
we infer that
$$
\mathcal{A}_{[-\omega_R,\omega_R]}(x_R) \leq 
\left(\sqrt{\frac{2m}{\alpha}}\frac{4}{2-\alpha}\right) R^{1-\alpha/2} + \sqrt{2} \left( M_+ + \frac{4C_+}{2+\alpha} \right).
$$
Therefore, the estimate from above in \eqref{stimalivello} holds for any 
\begin{equation}\label{st1}
M > \sqrt{2} \left( M_+ + \frac{4C_+}{2+\alpha} \right). 
\end{equation}
\smallbreak
Finally, we prove the estimate from below in \eqref{stimalivello}. As a first step, we prove that for any 
$u \in \Gamma_{R \xi^\pm}$ satisfying
\begin{equation}\label{minimok}
\min_t \vert u(t) \vert \leq K
\end{equation}
it holds that
\begin{equation}\label{stimamau}
\sqrt{\mathcal{M}_0(u)} \geq \sqrt{\frac{m}{\alpha}} \frac{4}{2-\alpha} R^{1-\alpha/2} - \frac{4C_+}{2+\alpha} - \sqrt{\frac{m}{\alpha}}\frac{4}{2-\alpha}K^{1-\alpha/2}.
\end{equation}
To prove this, we first observe that \eqref{minimok} implies the existence of
$t_1,t_2 \in (-1,1)$ such that $\vert u(t_i) \vert = K$
and $\vert u(t) \vert \geq K$ for $t \in [-1,t_1] \cup [t_2,1]$.
Now, we introduce the notation
$$
\mathcal{L}(u) = \int_{-1}^1 \vert \dot u(t) \vert \sqrt{V(u(t))}\,dt;
$$
writing $r(t) = \vert u(t) \vert$, we obtain
\begin{align*}
\sqrt{\mathcal{M}_0(u)} & \geq \mathcal{L}(u) \geq \int_{[-1,t_1] \cup [t_2,1]} \vert \dot u(t) \vert \sqrt{V(u(t))} \,dt \\
& \geq 
\sqrt{\frac{m}{\alpha}}\int_{[-1,t_1] \cup [t_2,1]} \vert \dot r(t) \vert r^{-\alpha/2}(t)\,dt
- C_+\int_{[-1,t_1] \cup [t_2,1]} \vert \dot r(t) \vert r^{-2-\alpha/2}(t)\,dt,
\end{align*} 
where the last inequality follows from \eqref{stimak}. Now, on one hand
$$
\int_{[-1,t_1] \cup [t_2,1]} \vert \dot r(t) \vert r^{-2-\alpha/2}(t)\,dt = 2 \int_K^R r^{-2-\alpha/2}\,dr \leq
\frac{4}{\alpha+2}.  
$$
On the other hand,
$$
\int_{-1}^{t_1} \vert \dot r(t) \vert r^{-\alpha/2}(t)\,dt
\geq \inf_{r \in \mathcal{R}} \int_{-1}^{t_1} \vert \dot r(t) \vert r^{-\alpha/2}(t)\,dt = \frac{2}{2-\alpha}(R^{1-\alpha/2}-K^{1-\alpha/2}),
$$
where $\mathcal{R} = \{ r \in H^1([-1,t_1]) : r(-1) = R, \,r(t_1) = K\}$, and analogous estimate holds for 
$
\int_{t_2}^{1} \vert \dot r(t) \vert r^{-\alpha/2}(t)\,dt$.
Summing up, \eqref{stimamau} is proved.

To conclude, we recall that that for any $h \in \Lambda$, there exists $s_h \in \mathbb{S}^1$ such that
$h(s_h)([-1,1]) \cap [c_1,c_2] \neq \emptyset$ (see \ref{segmento}); in particular, $h(s_h)$ satisfies \eqref{minimok}.  Hence
$$
\sup_{s \in \mathbb{S}^1} \sqrt{\mathcal{M}_0(h(s))} \geq \sqrt{\frac{m}{\alpha}} \frac{4}{2-\alpha} R^{1-\alpha/2} - \frac{4C_+}{2+\alpha} - \sqrt{\frac{m}{\alpha}}\frac{4}{2-\alpha}K^{1-\alpha/2}. 
$$
Recalling the definition of $c_0$ given in \eqref{deflivello}, the fact $\mathcal{M}_0(u_R) = c_0$ (compare with \eqref{livellozero}) and
\eqref{azione}, we obtain
$$
\mathcal{A}_{[-\omega_R,\omega_R]}(x_R) \geq \sqrt{\frac{2m}{\alpha}} \frac{4}{2-\alpha} R^{1-\alpha/2} - \sqrt{2}\left(\frac{4C_+}{2+\alpha} + \sqrt{\frac{m}{\alpha}}\frac{4}{2-\alpha}K^{1-\alpha/2}\right).
$$
Hence, the estimate from below in \eqref{stimalivello} holds for any 
\begin{equation}\label{st2}
M > \sqrt{2}\left(\frac{4C_+}{2+\alpha} + \sqrt{\frac{m}{\alpha}}\frac{4}{2-\alpha}K^{1-\alpha/2}\right).
\end{equation}
\smallbreak
Combining \eqref{st1} and \eqref{st2}, we conclude.

\section{Proof of the main result}\label{s5}

In this section we prove that the parabolic solutions $x_R$ given by Theorem \ref{thfix} satisfy the assumptions of 
Proposition \ref{approssimazione}, namely, \eqref{limitesup}, \eqref{limiteinf} and \eqref{stimatempi}.
In this way, we obtain the thesis of Theorem \ref{thmain}.

\subsection{Proof of \eqref{stimatempi}}\label{sectempi}

Of course, we assume here that $\min_t \vert x_R(t) \vert < K$ and we take $t^-_R, t^+_R$ as in
\eqref{stimatempi}. Notice that, by Lemma \ref{lemlj}, $\vert x_R(t) \vert \leq K$ for any 
$t \in [t^-_R,t^+_R]$ so that 
$$
V(x_R(t)) \geq V_K := \inf_{\vert x \vert \leq K} V(x) > 0, \quad \mbox{ for every } t \in [t^-_R,t^+_R].
$$
Then, using the conservation of the energy we can estimate $\mathcal{A}_{[-\omega_R,\omega_R]}(x_R)$ as follows:
\begin{align*}
\mathcal{A}_{[-\omega_R,\omega_R]}(x_R) & = \int_{[-\omega_R,t^-_R] \cup [t^+_R,\omega_R]}2V(x_R(t))\,dt + \int_{[t^-_R,t^+_R]}2V(x_R(t))\,dt \\
& \geq \int_{[-\omega_R,t^-_R] \cup [t^+_R,\omega_R]}\vert \dot x_R(t) \vert \sqrt{V(x_R(t))}\,dt + 2\left( t^+_R- t^-_R \right)V_K.
\end{align*}
Now, arguing as in the proof of \eqref{stimamau} we can see that 
$$
\int_{[-\omega_R,t^-_R] \cup [t^+_R,\omega_R]}\vert \dot x_R(t) \vert \sqrt{V(x_R(t))}\,dt \geq \sqrt{\frac{2m}{\alpha}} \frac{4}{2-\alpha} R^{1-\alpha/2} - \frac{M}{\sqrt{2}},
$$
so that
$$
2\left( t^+_R- t^-_R \right)V_K \leq \mathcal{A}_{[-\omega_R,\omega_R]}(x_R) - \sqrt{\frac{2m}{\alpha}} \frac{4}{2-\alpha} R^{1-\alpha/2} + \frac{M}{\sqrt{2}}.
$$
Recalling the estimate from above in \eqref{stimalivello}, we conclude.

\subsection{Proof of \eqref{limiteinf}}\label{limiteinfsec}
 
By contradiction, assume that, for instance,
\begin{equation}\label{deltar}
\delta_R := \min_t \vert x_R(t) - c_1 \vert = \vert x_R(\tau_R) - c_1 \vert \to 0^+.
\end{equation}
Setting $J_R(t) = \tfrac{1}{2}\vert x_R(t) - c_1 \vert^2$, we can perform computations analogous to the ones leading to \eqref{eqenergia} and
\eqref{eqpolare}; in particular, writing $V$ as in \eqref{vsing} and using \eqref{deltastar2}, we can easily see that
$$
\ddot J_R(t) > 0, \quad \mbox{ whenever} \quad \vert x_R(t) - c_1 \vert \leq \delta^*.
$$
Then, there exist $\tau^-_R,\tau^+_R$ such that 
$\tau^-_R < \tau_R < \tau^+_R$ and
$$
\vert x_R(\tau_R^\pm) - c_1 \vert = \delta^*
\quad \mbox{ and } \quad
\vert x_R(t) - c_1 \vert \leq \delta^*, 
\quad \mbox{ for any } t \in [\tau^-_R,\tau^+_R];
$$ 
moreover, for $r_R(t) := \vert x_R(t) - c_1 \vert$ it holds that
$\dot r_R(t) < 0$ for $t \in (\tau^-_R,\tau_R)$ and
$\dot r_R(t) > 0$ for $t \in (\tau_R,\tau^+_R)$.
As a consequence, using the conservation of the energy and \eqref{deltastar3}, we obtain
\begin{align*}
\tau^+_R - \tau_R & \geq \int_{\tau_R}^{\tau^+_R}\frac{\dot r_R(t)}{\sqrt{2V(x_R(t))}}\,dt \geq \sqrt{\frac{\alpha}{3m_1}}
\int_{\tau_R}^{\tau^+_R}\dot r_R(t)r_R(t)^{\alpha/2}\,dt \\
& =  \sqrt{\frac{\alpha}{3m_1}}\frac{2}{2+\alpha} \left( \delta_*^{1+\alpha/2} - \delta_R^{1+\alpha/2} \right),
\end{align*}
implying that $\tau^+_R - \tau_R$ is bounded away from zero for $R$ large; 
of course, the same holds for
$\tau_R - \tau^-_R$.
 
As a next step, we define
the function $v_R$ as
$$
v_R(t) = \frac{1}{\delta_R}\left(x_R\left( \delta_R^{1+\alpha/2}t + \tau_R \right) - c_1 \right), \qquad t \in [-\gamma_R,\sigma_R],
$$
where
$$
-\gamma_R = \frac{\tau^-_R - \tau_R}{\delta_R ^{1+\alpha/2}} \qquad 
\mbox{and} \qquad  
\sigma_R = \frac{\tau^+_R - \tau_R}{\delta_R^{1+\alpha/2}}.
$$
Notice that $\vert v_R(0) \vert = 1$, $\vert v_R(t) \vert \geq 1$ and $\vert \delta_R v_R(t) + c_1 \vert \leq \delta^*$ for
$t \in  [-\gamma_R,\sigma_R]$. The function $v_R$ satisfies
$$
\ddot v_R = - \frac{m_1 v_R}{\vert v_R \vert^{\alpha+2}} +\delta_R^{1+\alpha} \nabla \Phi_1(\delta_R v_R + c_1)
$$
and
$$
\frac{1}{2}\vert \dot v_R \vert^2 = \frac{m_1}{\alpha \vert v_R \vert^{\alpha}} + \delta_R^{\alpha} \Phi_1(\delta_R v_R + c_1).
$$ 
Moreover, in view of the above discussion, $-\gamma_R \to -\infty$ and $\sigma_R \to +\infty$. In view of these facts, it is easy to see
that $v_R \to v_\infty$ for $R \to +\infty$ in $\mathcal{C}^2_{\textnormal{loc}}(\mathbb{R})$, with $v_\infty$ an entire parabolic solution of the problem
$$
\ddot v_\infty = - \frac{m_1 v_\infty}{\vert v_\infty \vert^{\alpha+2}}.
$$
\smallbreak
We now continue the proof by showing that, as a consequence of
the above blow-up analysis,
\begin{equation}\label{morseblowup}
\liminf_{R \to +\infty} \textnormal{j}\left( d^2 \mathcal{A}_{[\tau^-_R,\tau^+_R]}(x_R)\right) \geq i(\alpha),
\end{equation}
with $i(\alpha)$ defined in \eqref{ialpha}.
Of course, $v_\infty$ is contained in a plane in $\mathbb{R}^3$
(say, $v_\infty(t) \cdot e \equiv 0$, for a suitable $e \in \mathbb{R}^3$); moreover, 
from \cite[Sections 3-4]{Tan93b} we know that for $L > 0$ large enough
there exist $i(\alpha)$ linearly independent $\varphi_1,\ldots,\varphi_{i(\alpha)} \in \mathcal{C}^\infty_c((-L,L);\mathbb{R})$ such that
\begin{equation}\label{autov}
\int_{-L}^L \left(
\dot \varphi_i^2(t) - \frac{m_1}{\alpha \vert v_\infty(t) \vert^{\alpha+2}}\varphi_i^2(t)\right)\,dt < 0, \qquad i=1,\ldots,i(\alpha).
\end{equation}
Notice that $\varphi_i \in \mathcal{C}^\infty_c((-\gamma_R.\sigma_R);\mathbb{R})$ for $R$ large.
We define, for $i=1,\ldots,i(\alpha)$ and $s \in [-\gamma_R,\sigma_R]$,
$$
k_{i}(s) = \varphi_i(s) e
$$
and, for $t \in [\tau_R^-,\tau_R^+]$,
$$
h_{i,R}(t) =  \delta_R \,\varphi_i\left( \frac{t-\tau_R}{\delta_R^{1+\alpha/2}} \right) e.
$$
An elementary computation shows that
\begin{align*}
\delta_R^{\alpha/2-1} & d^2 \mathcal{A}_{[\tau^-_R,\tau^+_R]}(x_R) [h_{i,R},h_{i,R}]
= \int_{-\gamma_R}^{\sigma_R} \left(\vert \dot k_{i,R} \vert^2 + \delta_R^{2+\alpha} D^2 V(\delta_R v_R + c_1)[k_{i,R},k_{i,R}]\right) \\
& = \int_{-\gamma_R}^{\sigma_R} \left(\vert \dot k_{i,R} \vert^2 -\frac{m_1}{\alpha \vert v_R \vert^{\alpha+2} \vert}\vert k_{i,R} \vert^2\right) \\
& \quad + \int_{-\gamma_R}^{\sigma_R} \left(  m_1 \frac{\left( v_R \cdot k_{i,R}\right)^2} {(\alpha+2) \vert v_R \vert^{\alpha+4}}
 + \delta_R^{2+\alpha} d^2 \Phi_1(\delta_R v_R + c_1)[k_{i,R},k_{i,R}]\right)
\end{align*}
Recalling that $v_\infty(t) \cdot e \equiv 0$ and passing to the limit, we easily obtain that
$$
\lim_{R \to +\infty} \delta_R^{\alpha/2-1} d^2 \mathcal{A}_{[\tau^-_R,\tau^+_R]}(x_R) [h_{i,R},h_{i,R}] = \int_{-L}^L \left(
\dot \varphi_i^2 - \frac{m_1}{\alpha \vert v_\infty \vert^{\alpha+2}}\varphi_i^2\right)
$$ 
which is negative in view of \eqref{autov}. This gives the desired conclusion \eqref{morseblowup}.
\smallbreak
In the case $\alpha > 1$, \eqref{morseblowup} immediately leads to a contradiction. Indeed, combining
\eqref{morsefix} together with the easy observation that
$\textnormal{j}( d^2 \mathcal{A}_{[\tau^-_R,\tau^+_R]}(x_R)) \leq \textnormal{j}\left( d^2 \mathcal{A}_{[-\omega_R,\omega_R]}(x_R)\right)$
yields a contradiction with $i(\alpha) > 1$.
\smallbreak
In the case $\alpha = 1$, more work is needed. At first, we observe that, arguing as in the proof of \eqref{stimatempi}
(see Section \ref{stimatempi}), we can prove that
$$
\int_{t^-_R}^{t^+_R} \vert \dot x_R(t) \vert^2 \leq M + \frac{M}{\sqrt{2}},
$$
so that, using \eqref{stimatempi}, $\int_{t^-_R}^{t^+_R} (\vert x_R \vert^2 + \vert \dot x_R \vert^2)$ is bounded as well. 
On the other hand, 
$\vert x_R(t^\pm_R) \vert$, $\vert \dot x_R(t^\pm_R) \vert = \sqrt{2V(x_R(t^\pm_R))}$ and
$$
\Vert \ddot x_R \Vert_{L^\infty([-\omega_R,t^-_R] \cup [t^+_R,\omega_R])}
$$
are also bounded. As a consequence,
defining $\tilde x_R$ as in \eqref{xrtraslata}, we have that there exists a function $x_\infty: \mathbb{R} \to \mathbb{R}^3$ 
such that $\tilde x_R \to x_\infty$ in $H^1_{\textnormal{loc}}(\mathbb{R})$ (in particular, uniformly on compact sets).
From \eqref{deltar} we deduce that there exists $t_\infty \in \mathbb{R}$ such that $x_\infty(t_\infty) = c_1$; moreover,
via a blow-up analysis analogous to the one leading to \eqref{morseblowup} (and recalling \eqref{morsefix}),    
we see that $x_\infty(t) \notin \Sigma$ for $t \neq t_\infty$. As a consequence, $x_\infty$ is a (one-collision) generalized parabolic solution of
\eqref{main} and, reasoning as in the proof of Claim 4 in Proposition \ref{approssimazione}, we obtain 
\begin{equation}\label{limitidiversi}
\lim_{t \to -\infty} \frac{x_\infty(t)}{\vert x_\infty(t) \vert} = \xi^- \neq \xi^+ = \lim_{t \to +\infty} \frac{x_\infty(t)}{\vert x_\infty(t) \vert}.
\end{equation}
On the other hand, we can argue exactly as in Section \ref{noncoll} (using regularization techniques) to prove that
$$
x_\infty(t) = x_\infty(2t_\infty-t), \quad \mbox{ for all } t \in \mathbb{R}.
$$
This clearly contradicts \eqref{limitidiversi}.

\subsection{Proof of \eqref{limitesup}}\label{limitesupsec}

By contradiction, assume that 
$$
\rho_R := \min_t \vert x_R(t) \vert = \vert x_R(\tau_R) \vert \to +\infty.
$$
(notice that here $\tau_R$ has a different meaning with respect to \eqref{deltar}). 
In particular, we can always suppose $\rho_R \geq K$; then, Lemma \ref{lemlj} and Corollary 
\ref{cortempi} are applicable and we obtain
\begin{equation}\label{tempifin}
\begin{array}{ll}
\vspace{0.2cm}
\displaystyle{\omega_R - \tau_R} & \geq \displaystyle{\frac{1}{(1+\alpha/2)\sqrt{2C_+}}\left( R^{1+\alpha/2}-\rho_R^{1+\alpha/2}\right)},
\\
\displaystyle{-\omega_R - \tau_R} & \geq \displaystyle{\frac{1}{(1+\alpha/2)\sqrt{2C_+}}\left( R^{1+\alpha/2}-\rho_R^{1+\alpha/2}\right)}.
\end{array}
\end{equation}
Let us set
$$
d_R = \frac{\rho_R}{R} \in (0,1], \qquad d = \lim_{R \to +\infty}d_R \in [0,1], 
$$
and we distinguish two cases.
\medbreak
If $d = 0$, we define
$$
v_R(t) = \frac{1}{\rho_R}\left(x_R\left( \rho_R^{1+\alpha/2}t + \tau_R \right) \right), \qquad t \in [-\gamma_R,\sigma_R],
$$
where
$$
-\gamma_R = \frac{-\omega_R - \tau_R}{\rho_R ^{1+\alpha/2}} \qquad 
\mbox{and} \qquad  
\sigma_R = \frac{\omega_R - \tau_R}{\rho_R^{1+\alpha/2}}.
$$
Notice that $\vert v_R(0) \vert = 1$, $1 \leq \vert v_R(t) \vert \leq R/\rho_R$ for
$t \in  [-\gamma_R,\sigma_R]$. Writing $V$ as in \eqref{vinf}, the function $v_R$ satisfies
$$
\ddot v_R = - \frac{m v_R}{\vert v_R \vert^{\alpha+2}} +\rho_R^{1+\alpha} \nabla W(\rho_R v_R)
$$
and
$$
\frac{1}{2}\vert \dot v_R \vert^2 = \frac{m}{\alpha \vert v_R \vert^{\alpha}} + \rho_R^{\alpha} W(\rho_R v_R).
$$
Moreover, from \eqref{tempifin} we obtain 
$$
\sigma_R = \frac{\omega_R - \tau_R}{\rho_R^{1+\alpha/2}} \geq \frac{1}{(1+\alpha/2)\sqrt{2C_+}}\frac{1-d_R^{1+\alpha/2}}
{d_R^{1+\alpha/2}} \to +\infty
$$
and, analogously, $-\gamma_R \to -\infty$. Finally, using \eqref{stimaW} we find
\begin{equation}\label{w1}
\vert \rho_R^{1+\alpha} \nabla W(\rho_R v_R) \vert \leq \rho_R^{1+\alpha} \frac{C_+}{\vert \rho_R v_R \vert^{\alpha+3}} \leq \frac{C_+}{\rho_R^2} \to 0
\end{equation}
and
\begin{equation}\label{w2}
\vert \rho_R^{\alpha} W(\rho_R v_R) \vert \leq \rho_R^{\alpha} \frac{C_+}{\vert \rho_R v_R \vert^{\alpha+2}} \leq \frac{C_+}{\rho_R^2} \to 0
\end{equation}
for $R \to +\infty$, uniformly in $t$. We can thus readily see
that $v_R \to v_\infty$ in $\mathcal{C}^2_{\textnormal{loc}}(\mathbb{R})$, with $v_\infty$ an entire parabolic solution of the problem
$$
\ddot v_\infty = - \frac{m v_\infty}{\vert v_\infty \vert^{\alpha+2}}.
$$
Moreover, following the arguments used in the proof of Claim 4 in Proposition \ref{approssimazione}, we also have
$$
\lim_{t \to -\infty} \frac{v_\infty(t)}{\vert v_\infty(t) \vert} = \xi^- \neq \xi^+ = \lim_{t \to +\infty} \frac{v_\infty(t)}{\vert v_\infty(t) \vert}
$$
This immediately gives a contradiction in the case $\alpha = 1$,
since, as well-known, the asymptotic directions of parabolic solutions of the Kepler problem must coincide (cf. Proposition \ref{propspan}).
 On the other hand, for $\alpha > 1$ we can argue as in Section \ref{limiteinfsec}
(using this time \eqref{w1} and \eqref{w2} to pass to the limit) 
 to prove that
 $$
 \liminf_{R \to +\infty} \textnormal{j}\left( d^2 \mathcal{A}_{[-\omega_R,\omega_R]}(x_R)\right) \geq i(\alpha)
 $$
 and thus contradicting \eqref{morsefix} since $i(\alpha) \geq 2$ for $\alpha > 1$. 
\medbreak
We now focus on the case $d \in (0,1]$. Let us define
$$
\tilde v_R(t) = \frac{1}{R}\left(x_R\left( R^{1+\alpha/2}t + \tau_R \right) \right), \qquad t \in [-\tilde\gamma_R,\tilde\sigma_R],
$$
where
$$
-\tilde\gamma_R = \frac{-\omega_R - \tau_R}{R ^{1+\alpha/2}} \qquad 
\mbox{and} \qquad  
\tilde\sigma_R = \frac{\omega_R - \tau_R}{R^{1+\alpha/2}}.
$$
The function $\tilde v_R$ satisfies
$$
\ddot{\tilde{v}}_R = - \frac{m \tilde v_R}{\vert \tilde v_R \vert^{\alpha+2}} +R^{1+\alpha} \nabla W(R \tilde v_R)
$$
and
$$
\frac{1}{2}\vert \dot{\tilde{v}}_R \vert^2 = \frac{m}{\alpha \vert \tilde v_R \vert^{\alpha}} + R^{\alpha} W(R \tilde v_R).
$$
Moreover, $\vert \tilde v_R(0) \vert = d_R$, $\tilde v_R(-\tilde\gamma_R) = \xi^-$, $\tilde v_R(\tilde\sigma_R) = \xi^+$ and $d_R \leq \vert \tilde v_R(t) \vert \leq 1$ for $t \in  [-\tilde\gamma_R,\tilde\sigma_R]$. Finally, similarly as in \eqref{w1} and \eqref{w2},
\begin{equation}\label{w3}
\vert R^{1+\alpha} \nabla W(R \tilde v_R) \vert \leq \frac{C_+}{R^2}\left(\frac{2}{d}\right)^{\alpha+2}, \qquad   
\vert R^{\alpha} W(R \tilde v_R) \vert \leq \frac{C_+}{R^2} \leq \frac{C_+}{R^2}\left(\frac{2}{d}\right)^{\alpha+3}
\end{equation}
for $R$ large enough.

We now claim that $\tilde\sigma_R + \tilde\gamma_R$ is bounded away from zero. 
Indeed, if $\tilde\sigma_R \to 0^+$ and $-\tilde\gamma_R \to 0^-$, then
from
\begin{equation}\label{lim1}
\xi^+ = \tilde v_R(\tilde\sigma_R) = \tilde v_R(0) + \int_0^{\tilde\sigma_R} \dot{\tilde{v}}_R(t)\,dt
\end{equation}
and
\begin{equation}\label{lim2}
\xi^- = \tilde v_R(-\tilde\gamma_R) = \tilde v_R(0) + \int_{-\tilde\gamma_R}^0 \dot{\tilde{v}}_R(t)\,dt,
\end{equation}
together with the fact that $\max_t \vert \dot{\tilde{v}}_R(t) \vert $ is bounded in $R$ in view of \eqref{w3}, we obtain
$\tilde v_R(0) \to \xi^+$  and $\tilde v_R(0) \to \xi^-$, which is not possible
since $\xi^+ \neq \xi^-$.

As a consequence, there exists a nontrivial interval $\tilde I_\infty = [-\tilde\gamma_\infty,\tilde\sigma_\infty]$ 
such that $\tilde v_R \to \tilde v_\infty$ in $\mathcal{C}^2_{\textnormal{loc}}(\tilde I_\infty)$;
moreover, $d \leq \vert \tilde v_\infty(t) \vert \leq 1$ for $t \in \tilde I_\infty$ and $\tilde v_\infty$ is a parabolic solution of
\begin{equation}\label{pblimite}
\ddot{\tilde{v}}_\infty = - \frac{m \tilde v_\infty}{\vert \tilde v_\infty \vert^{\alpha+2}}.
\end{equation}
This is possible only if $\tilde I_\infty$ is a compact interval (compare with the discussion before Proposition \ref{propspan});
as a consequence, the $\mathcal{C}^2_{\textnormal{loc}}$ convergence actually reduces to the $\mathcal{C}^2$ one.
Summing up, and passing to the limit in \eqref{lim1} and \eqref{lim2}, 
$\tilde v_\infty$ is a parabolic solution of the (free-time) fixed-endpoints problem
$$
\left\{
\begin{array}{l}
\vspace{0.2cm}
\displaystyle{\ddot{\tilde{v}}_\infty = - \frac{m \tilde v_\infty}{\vert \tilde v_\infty \vert^{\alpha+2}}}\\
\tilde v_\infty(T_1) = \xi^-, \; \tilde v_\infty(T_2) = \xi^+
\end{array}
\right.
$$
with $T_1 = -\tilde\gamma_\infty$ and $T_2 = \tilde\sigma_\infty$.

Now, using the fact that $x_R$ has zero energy, we write
\begin{align*}
\mathcal{A}_{[-\omega_R,\omega_R]}(x_R) & = \int_{-\omega_R}^{\omega_R} 2 V(x_R(t))\,dt = 2 R^{1+\alpha/2} \int_{-\tilde\sigma_R}^{\tilde\gamma_R} V(R \tilde v_R(s))\,ds \\
& = 2 R^{1-\alpha/2} \int_{-\tilde\sigma_R}^{\tilde\gamma_R} \left( \frac{m}{\alpha \vert \tilde v_R(s) \vert^\alpha} + R^\alpha W(R \tilde v_R(s))\right) \,ds
\end{align*}
so that, using \eqref{w3},
$$
\lim_{R \to +\infty} \frac{\mathcal{A}_{[-\omega_R,\omega_R]}(x_R)}{R^{1-\alpha/2}} = 2 
\int_{-\tilde\sigma_\infty}^{\tilde\gamma_\infty} \frac{m}{\alpha \vert \tilde v_\infty(s) \vert^\alpha}\,ds.
$$
Using Proposition \ref{azionebolza}
$$
\int_{-\tilde\sigma_\infty}^{\tilde\gamma_\infty} \frac{m}{\alpha \vert \tilde v_\infty(s) \vert^\alpha}\,ds = \mathcal{A}^{\alpha,m}_{[-\tilde\gamma_\infty,\tilde\sigma_\infty]}(\tilde v_\infty) < \sqrt{\frac{2m}{\alpha}}\frac{4}{2-\alpha},
$$
so that a contradiction with \eqref{stimalivello} is obtained.

\begin{remark}\label{xi}
When $\xi^+ = \xi^-$, the arguments developed along the paper can be adapted to prove the existence of a generalized (see Section \ref{secgen}) spatial parabolic solutions of \eqref{main} satisfying \eqref{asintoticast} (for $\xi^+ = \xi^-$).

Indeed, we first observe that a variant of Theorem \ref{thfix} can be proved for $\xi^+ = \xi^-$, giving the existence of a generalized parabolic solution
of \eqref{bol} satisfying the level estimate \eqref{stimalivello}. 
This can be done via an approximation argument for $\xi_n^+ \to \xi^+$ and $\xi_n^- \to \xi^-$ (with $\xi_n^- \neq \xi_n^-$), the convergence for $n \to +\infty$ of the corresponding solution coming from \eqref{stimalivello} (with some care, it is possible to see that the constant $M$ can be chosen independently on $n$).
%, while the Morse-index bound \eqref{morsefix} ensures the fact that the number of collision of the limit solution is at most one (via the usual blow-up argument). 

Second, we pass to the limit $R \to +\infty$ following the steps in the proof of Proposition \ref{approssimazione}.
Minor variants are needed, since just $H^1_{\textnormal{loc}}$ convergence is possible near the collision instants; however, a careful use
of the action estimate \eqref{stimalivello} allows us to obtain the conclusion. We leave the details to the reader for the sake of briefness.
\end{remark}

\section{Appendix: the $\alpha$-homogeneous problem}

In this final section we collect some useful results about parabolic solutions of the $\alpha$-homogeneous problem
\begin{equation}\label{aom}
\ddot x = - \frac{\mu x }{\vert x \vert^{\alpha+2}}, \qquad x \in \mathbb{R}^3 \setminus \{0\},
\end{equation}
where $\mu > 0$ and $\alpha \in [1,2)$. Of course, the term \emph{parabolic} is here meant with respect to the natural energy associated to \eqref{aom}, namely
$x$ is a parabolic solution
 of \eqref{aom} if $\tfrac{1}{2}\vert \dot x(t) \vert^ 2 = \frac{\mu}{\alpha \vert x(t) \vert^\alpha}$.
 
It is well-known that any solution to \eqref{aom} is contained in a plane; therefore, without loss of generality we assume that $x \in \mathbb{R}^2$ and we use polar coordinates
$$
x(t) = r(t)e^{i\theta(t)}, \qquad r(t) > 0.
$$
Recall also that any solution $x: I \to \mathbb{R}^2 \setminus \{0\}$ (with $I \subset \mathbb{R}$ interval) to \eqref{aom} has constant angular momentum, that is (in polar coordinates)
\begin{equation}\label{momang}
r^2(t) \dot\theta(t) \equiv c, \qquad t \in I, \quad \mbox{ for some } \, c \in \mathbb{R}.
\end{equation}
In particular, either the function $t \mapsto \theta(t)$ is constant ($c=0$) or it is strictly monotone ($c \neq 0$).
Combining \eqref{momang} with the fact that $x$ has zero energy, we obtain
\begin{equation}\label{ene}
\frac{1}{2} \dot r^2(t) = \frac{\mu}{\alpha r^\alpha(t)} - \frac{1}{2} \frac{c^2}{r^2(t)}, \qquad t \in I.
\end{equation}
Finally, the Lagrange-Jacobi identity (compare with \eqref{lj}) reads as
\begin{equation}\label{lj2}
\frac{d^2}{dt^2}\left( \frac{1}{2}r^2(t)\right) = \frac{(2-\alpha)\mu}{\alpha r^{\alpha}(t)}, \qquad t \in I.
\end{equation}

The case of parabolic solutions with zero angular momentum is easily discussed. Indeed, by integrating \eqref{ene} for $c = 0$ we find that
the only solutions are of the type
$$
x(t) = \gamma_{\alpha,\mu} (t-t_0)^{\frac{2}{2+\alpha}} e^{i\theta_0}, \qquad t \in I = (t_0,+\infty),
$$
$$
x(t) = \gamma_{\alpha,\mu} (t_0-t)^{\frac{2}{2+\alpha}} e^{i\theta_0}, \qquad t \in I = (-\infty,t_0),
$$
for $t_0 \in \mathbb{R}$, $\theta_0 \in [0,2\pi)$, where 
$$
\gamma_{\alpha,\mu} = \left(\sqrt{\frac{\mu}{2\alpha}} (2+\alpha)\right)^{\frac{2}{2+\alpha}}
$$
as already defined in \eqref{gamdef}. In particular, there are no entire rectilinear parabolic solutions of \eqref{aom}.
\smallbreak
From now, we thus consider the case of solutions with non-zero angular momentum. 
First, we deal with entire parabolic solutions to \eqref{aom}. From the Lagrange-Jacobi identity \eqref{lj2} we deduce that there exists
$t_* \in \mathbb{R}$ such that $\dot r(t) < 0$ for $t < t_*$ and $\dot r(t) > 0$ for $t > t_*$; moreover,
$r(t) \to +\infty$ for $\vert t \vert \to +\infty$ (compare with \cite[Lemma 7.6]{BarTerVer14}).
We also have the following.

\begin{proposition}\label{propspan}
Let $x: \mathbb{R} \to \mathbb{R}^3 \setminus \{0\}$ be a parabolic solution of \eqref{aom} (with angular momentum $c \neq 0$).
Then
\begin{equation}\label{romo}
r(t) \sim \gamma_{\alpha,\mu} \vert t \vert^{\frac{2}{2+\alpha}}, \qquad \vert t \vert \to +\infty,
\end{equation}
and the limits $\theta(\pm\infty) := \lim_{t \to \pm \infty}\theta(t)$ satisfy
\begin{equation}\label{span}
\vert \theta(+\infty) - \theta(-\infty) \vert = \frac{2\pi}{2-\alpha}.
\end{equation}
%\item there exist $i(\alpha)$ linearly independent functions 
%$\varphi_1,\ldots,\varphi_{i(\alpha)} \in \mathcal{C}_c^\infty(\mathbb{R},\mathbb{R})$
%such that
%$$
%\int_\mathbb{R} \left(\dot \varphi_i^2(t) - \frac{\mu}{\alpha r(t) \vert^{\alpha+2}}\varphi_i^2(t)\right)\,dt < 0, \qquad i=1,\ldots,i(\alpha).
%$$
\end{proposition}

We observe that the asymptotic estimate \eqref{romo} follows from \eqref{rasy}; however, in this simpler setting we can provide a slightly more direct proof.  
We also notice that, for $\alpha = 1$, \eqref{span} gives $\vert \theta(+\infty)-\theta(-\infty) \vert = 2\pi$, according to the fact 
that $t \mapsto x(t) = r(t)e^{i\theta(t)}$ parameterizes a parabola in the plane. 
On the other hand, for $\alpha > 1$, i) and ii) imply that $x(t)$ is a self-intersecting planar path, with exactly
$$
i_*(\alpha) = \max \left\{ k \in \mathbb{N} \, : \, k < \frac{1}{2-\alpha}\right\}
$$
self-intersection. Notice that this quantity is strictly related to the 
constant $i(\alpha)$ defined in \eqref{ialpha}.

\begin{proof}
We define the function
$$
\Gamma(t) = r^{\alpha}(t)\dot r^2(t).
$$
Using \eqref{ene} we obtain
$$
\Gamma(t) = \frac{2\mu}{\alpha} - c^2 r^{\alpha-2}(t),
$$
so that
$$
\lim_{\vert t \vert \to +\infty}\Gamma(t) = \frac{2\mu}{\alpha}. 
$$
Hence $r^{\alpha/2}(t)\dot r(t) \to \pm \sqrt{\tfrac{2\mu}{\alpha}}$ for $t \to \pm\infty$ and we obtain
the asymptotic estimate for $r$ using de l'Hopital rule.

To prove \eqref{span}, we assume (to fix the ideas) that $\dot\theta(t) > 0$ and let $r_* = r(t_*)$. Using \eqref{momang} and \eqref{ene}, we have
\begin{align*}
\theta(+\infty) - \theta(t_*) & = c \int_{t_*}^\infty \frac{dt}{r^2(t)} =
c \int_{t_*}^\infty \frac{\dot r(t)}{r^2(t)\sqrt{\frac{2\mu}{\alpha r^\alpha(t)} - \frac{c^2}{r^2(t)}}}\,dt
\\ & = \int_{0}^{\frac{1}{r_*}} \frac{d\xi}{\sqrt{\frac{2\mu}{\alpha c^2}\xi^\alpha - \xi^2}} = 
\frac{2}{2-\alpha}\int_{0}^{r_*^{(\alpha-2)/2}} \frac{d\eta}{\sqrt{\frac{2\mu}{\alpha c^2} - \eta^2}}, 
\end{align*} 
where in the last equality we have used the change of variable $\xi = \eta^{2/(2-\alpha)}$.
From \eqref{ene} with $t = t_*$ we find
$
r_*^{\alpha-2} = \tfrac{2\mu}{\alpha c^2}, 
$
so that
$$
\int_{0}^{r_*^{(\alpha-2)/2}} \frac{d\eta}{\sqrt{\frac{2m}{\alpha c^2} - \eta^2}} = \frac{\pi}{2}
$$
and, therefore,
$$
\theta(+\infty) - \theta(t_*) = \frac{\pi}{2-\alpha}.
$$
Evaluating in an analogous way $\theta(t_*) - \theta(-\infty)$, we conclude.
\end{proof}

We now look for parabolic solutions of the (free-time) fixed-endpoints problem
\begin{equation}\label{freeapp2}
\left\{
\begin{array}{l}
\vspace{0.1cm}
\displaystyle{\ddot x = - \frac{\mu x }{\vert x \vert^{\alpha+2}}} \\
%\vspace{0.1cm}
%\displaystyle{\frac{1}{2}\vert \dot x \vert^2 = \frac{\mu}{\alpha \vert x \vert^{\alpha}}} \\
%\vspace{0.1cm}
x(T_1) = x_1, \quad x(T_2) = x_2,  
\end{array}
\right.
\end{equation}
where $x_1, x_2 \in \mathbb{R}^2$. Our aim is to prove the following result.

\begin{proposition}\label{azionebolza}
Let $x_1, x_2 \in \mathbb{R}^2$ be such that $x_1 \neq x_2$ and $\vert x_1 \vert = \vert x_2 \vert = 1$. If $x$ is a parabolic solution of problem \eqref{freeapp2}, 
then
$$
\mathcal{A}^{\alpha,\mu}_{[T_1,T_2]}(x) < \sqrt{\frac{2\mu}{\alpha}}\frac{4}{2-\alpha},
$$
where $\mathcal{A}^{\alpha,\mu}_{[T_1,T_2]}(x) = \int_{T_1}^{T_2} \left( \tfrac{1}{2}\vert \dot x \vert^2 + \frac{\mu}{\alpha \vert x \vert^{\alpha}}\right)$ is the action functional associated with \eqref{aom}.
\end{proposition}

The proof of Proposition \ref{azionebolza} will be based on the fact that solutions of problem \eqref{freeapp2} can be classified
according to their homotopy class in the punctured plane $\mathbb{R}^2 \setminus \{0\}$. Precisely, defining the rotation index
$\textnormal{Rot}_{[T_1,T_2]}(x)$ of the path $t \mapsto x(t) = r(t)e^{i\theta(t)}$
as
$$
\textnormal{Rot}_{[T_1,T_2]}(x) = \frac{\theta(T_2)-\theta(T_1)}{2\pi}, 
$$
it is clear that any solution of \eqref{freeapp2} satisfies
\begin{equation}\label{rot}
\textnormal{Rot}_{[T_1,T_2]}(x) = \frac{\theta_2-\theta_1}{2\pi} + l
\end{equation}
for some $l \in \mathbb{Z}$, where $x_i = e^{i\theta_i}$, $\theta_i \in [0,2\pi)$, $i=1,2$.

An existence and uniqueness result for parabolic solutions of \eqref{freeapp2} with prescribed rotation index is given in the Proposition below.

\begin{proposition}\label{pianofasi}
Let $x_1, x_2 \in \mathbb{R}^2$ be such that $x_1 \neq x_2$ and $\vert x_1 \vert = \vert x_2 \vert = 1$ and let 
$l \in \mathbb{Z}$. Then, problem \eqref{freeapp2} has a parabolic solution
satisfying \eqref{rot} if and only if
\begin{equation}\label{elle}
\vert \theta_2 - \theta_1 + 2\pi l \vert < \frac{2\pi}{2-\alpha}
\end{equation}
%\begin{equation}\label{elle}
%\begin{cases}
%\displaystyle{\theta(x_1,x_2) + 2\pi l \in \left(0,\frac{2\pi}{2-\alpha}\right)} \quad \mbox{ if } \, l \geq 0 \\
%\displaystyle{\theta(x_1,x_2) + 2\pi l \in \left(-\frac{2\pi}{2-\alpha},0\right)} \quad \mbox{ if } \, l \leq -1
%\end{cases}
%\end{equation}
and, in this case, the solution is unique (up to a time-translation).
\end{proposition}

Based on this, we can give a proof of Proposition \ref{azionebolza}.

\begin{proof}[Proof of Proposition \ref{pianofasi}]
Assume that $x_*$ is a parabolic solution \eqref{freeapp2}. Then, $x_*$ satisfies \eqref{rot} for some $l \in \mathbb{Z}$ and, in view of Proposition \ref{pianofasi}, $l$ fulfills \eqref{elle}. Define
$$
\tilde K_l = \left\{ u \in H^1([-1,1];\mathbb{R}^2\setminus \{0\}) \, : \, 
\begin{array}{l}
\vspace{0.2cm}
u(-1) = x_1, \, u(1) = x_2 \\
\textnormal{Rot}_{[-1,1]}(u) = \frac{\theta_2-\theta_1}{2\pi} + l   
\end{array}
\right\} 
$$
and let $K_l$ be the closure of $\tilde K_l$ in the weak topology of $H^1$. We consider the minimization problem
\begin{equation}\label{pbminimo}
\min_{u \in K_l} \mathcal{I}(u)
\end{equation}
where $\mathcal{I}(u) = \int_{-1}^1 \vert \dot u \vert^2 \int_{-1}^1 \tfrac{\mu}{\alpha \vert u \vert^\alpha}$ is the zero-energy Maupertuis functional
associated to \eqref{aom} (we assume throughout this proof that the reader is familiar with the theory of the Maupertuis functional, as described for instance in \cite[Appendix B]{SoaTer13}). It is easy to see (compare with Lemma \ref{mauplem}) that the minimization problem \eqref{pbminimo} has a solution. The crucial point is that, since $l$ satisfies \eqref{elle}, we know from \cite[Corollary 1.11]{SoaTer13} that any minimum pont is collision-free and, hence, belongs to $\tilde K_l$. Therefore, a suitable rescaling solves problem \eqref{freeapp2}. By the uniqueness property in Proposition \ref{pianofasi}, we conclude that the minimization problem \eqref{pbminimo} has a unique solution $u_*$ which is nothing but a rescaling of $x_*$.
In particular,
$$
\mathcal{I}(u_*) < \mathcal{I}(u), \quad \mbox{ for any } u \in K_l.
$$
Now, on one hand $\mathcal{I}(u_*) = \tfrac{1}{2}\left(\mathcal{A}^{\alpha,\mu}_{[T_1,T_2]}(x_*)\right)^2$ (compare with \eqref{azione}). 
On the other hand, defining 
$$
u(t) = \begin{cases}
(-t)^{\frac{2}{2+\alpha}}\, x_1 & \quad \mbox{ if } \, t \in [-1,0] \\
t^{\frac{2}{2+\alpha}}\, x_2 & \quad \mbox{ if } \, t \in [0,1]
\end{cases}
$$ 
it is easy to see that
$u \in K_l$ and
$$
\mathcal{I}(u) = \frac{1}{2}\left( \sqrt{\frac{2\mu}{\alpha}}\frac{4}{2-\alpha}\right)^2,
$$
concluding the proof.
\end{proof}

We conclude the section by proving Proposition \ref{pianofasi}.

\begin{proof}[Proof of Proposition \ref{pianofasi}]
The fact that condition \eqref{elle} is necessary follows from \eqref{span}, recalling the fact that,
for a (non-rectilinear) parabolic solution $x(t) = e^{i\theta(t)}$, the function $t \mapsto \theta(t)$ is strictly monotone. We now focus on the existence and uniqueness of a parabolic solution of \eqref{freeapp2}-\eqref{rot} 
when \eqref{elle} is satisfied; without loss of generality, we will also take $T_1 = -T$ and 
$T_2 = T$, with $T > 0$ to be determined. 

At first, we observe that $x: [-T,T] \to \mathbb{R}^2 \setminus \{0\}$ is a parabolic solution of \eqref{freeapp2}
with angular momentum equal to $c$ (see \eqref{momang}) if and only if 
$$
y(t) = \frac{1}{\vert c \vert^{\frac{2}{2-\alpha}}}\, x \left( \vert c \vert^{\frac{\alpha+2}{2-\alpha}}\,t \right), \qquad 0 \leq \vert t \vert \leq T_c := \vert c \vert^{-\frac{\alpha+2}{2-\alpha}}\, T,  
$$ 
is a parabolic solution of
\begin{equation}\label{freeapp3}
\left\{
\begin{array}{l}
\vspace{0.1cm}
\displaystyle{\ddot y = - \frac{\mu y }{\vert y \vert^{\alpha+2}}} \\
%\vspace{0.2cm}
%\displaystyle{\frac{1}{2}\vert \dot y \vert^2 = \frac{\mu}{\alpha \vert y \vert^{\alpha}}} \\
%\vspace{0.1cm}
\displaystyle{y(-T_c) = \frac{e^{i\theta_1}}{\vert c \vert^{\frac{2}{2-\alpha}}}, \quad y(T_c) = \frac{e^{i\theta_2}}{\vert c \vert^{\frac{2}{2-\alpha}}}},  
\end{array}
\right.
\end{equation}
with angular momentum equal to $\sgn(c)$; moreover, 
$\textnormal{Rot}_{[-T,T]}(x) = \textnormal{Rot}_{[-T_c,T_c]}(y)$.
Passing to polar coordinates $y(t) = \rho(t)e^{i\varphi(t)}$, with $\varphi(-T_c) = \theta_1$,
it is easy to see that this is equivalent to the equations
\begin{equation}\label{freeapp4}
\left\{
\begin{array}{ll}
\vspace{0.1cm}
\displaystyle{\frac{1}{2} \dot \rho^2 + F(\rho) = 0}, & \qquad \displaystyle{F(\rho) = \frac{\mu}{\alpha\rho^\alpha} - \frac{1}{2\rho^2}}, \\
\vspace{0.1cm}
\displaystyle{\dot \varphi = \frac{\sgn(c)}{\rho^2},}
\end{array}
\right.
\end{equation}
together with the boundary conditions
\begin{equation}\label{rhoc}
\rho(-T_c) = \rho(T_c) = \frac{1}{\vert c \vert^{\frac{2}{2-\alpha}}}
\end{equation}
and
\begin{equation}\label{phic}
\varphi(T_c) = \theta_2 + 2 \pi l.
\end{equation}

Let us define
$$
\rho_* = \left( \frac{\alpha}{2\mu} \right)^{\frac{1}{2-\alpha}},
$$
that is, $\rho_*$ is the unique point such that $F(\rho_*) = 0$. A simple phase-plane argument shows that
there exists a unique solution $\rho_*$ of the first-order differential equation $\frac{1}{2} \dot \rho^2 + F(\rho) = 0$ satisfying
$\rho_*(0) = \rho_*$; moreover, $\rho_*$ is an even function defined on the whole real line. Hence, we easily see that the first equation in \eqref{freeapp4} has a solution $\rho_c$ satisfying the boundary condition \eqref{rhoc} 
if and only if
$$
c^2 < \frac{2\mu}{\alpha}.
$$
In this case, $\rho_c(t) = \rho_*(t)$ for $t \in [-T_c,T_c]$, where
\begin{equation}\label{tc}
T_c = \sqrt{2} \int_{\rho_*}^{1/\vert c \vert^{\frac{2}{2-\alpha}}} \frac{dr}{\sqrt{-F(r)}}. 
\end{equation}
On the other hand, integrating the second equation and imposing the boundary condition \eqref{phic} we obtain
\begin{equation}\label{shooting}
\Theta(c) = \theta_2 - \theta_1 + 2\pi l,
\end{equation}
where we have set
$$
\Theta(c) = \sgn(c) \int_{-T_c}^{T_c} \frac{ds}{\rho_*^2(s)} = 2 \,\sgn(c) \int_0^{T_c}\frac{ds}{\rho_*^2(s)}, \qquad c \in \left(-\frac{2\mu}{\alpha}, \frac{2\mu}{\alpha}\right) \setminus \{0\}. 
$$

Now, recalling \eqref{tc} we immediately see that $\Theta$ is strictly decreasing on 
$\left(-\frac{2\mu}{\alpha},0\right)$ and on $\left(0,\frac{2\mu}{\alpha}\right)$, with
$$
\lim_{c \to \pm \frac{2\mu}{\alpha}}\Theta(c) = 0.
$$
On the other hand, $T_c \to +\infty$ for $c \to 0$ and 
$$
\lim_{c \to 0^{\pm}}\Theta(c) = \pm 2 \int_{0}^{+\infty} \frac{ds}{\rho_*^2(s)} = \pm \frac{2\pi}{2-\alpha} 
$$
as already shown along the proof of \eqref{span}.
Hence, \eqref{shooting} is uniquely solvable if and only if $0 \neq \vert \theta_2 - \theta_1 + 2\pi l \vert < \tfrac{2\pi}{2-\alpha}$,
which is precisely the assumption \eqref{elle} (notice that $\theta_2 - \theta_1 + 2\pi l \neq 0$ since $x_1 \neq x_2$).
\end{proof}

%\subsection*{Acknowledgments}
%
%The authors wish to thank Prof. Rafael Ortega for some useful discussions.  
%\smallbreak
%\noindent
%Work partially supported by the 
%ERC Advanced Grant 2013 n. 339958
%``Complex Patterns for Strongly Interacting Dynamical Systems - COMPAT''
%(A.B, W.D. and S.T.), by the PRIN-2012-74FYK7 Grant ``Variational and perturbative aspects of nonlinear differential problems'' (W.D. and S.T.)
%and by the GNAMPA Project 2015 ``Equazioni Differenziali
%Ordinarie sulla retta reale'' (A.B. and W.D.).
%Conflict of Interest: The authors declare that they have no conflict of interest.

\smallbreak
\noindent
{\bf Conflict of Interest.} The authors declare that they have no conflict of interest.

\end{document}